\documentclass[12pt,reqno]{elsarticle}
\usepackage{a4wide}
\usepackage{amsmath}
\usepackage[italian, english]{babel}
\usepackage{amsfonts}
\usepackage{amssymb}
\usepackage{dsfont}
\usepackage{mathrsfs}
\usepackage{graphicx}
\usepackage{fancyhdr}
\usepackage{multicol}
\usepackage{enumitem}
\usepackage{amsthm}
\usepackage{cancel}
\usepackage{empheq,etoolbox}
\usepackage{cases}
\usepackage[all]{xy}
\usepackage{stmaryrd}
\usepackage[colorlinks, citecolor=blue, linkcolor=red]{hyperref}
\usepackage{tikz,tikz-cd}
\usepackage{accents}
\def\RR{{\mathbb R}}

\tikzset{
	subset/.style={
		draw=none,
		edge node={node [sloped, allow upside down, auto=false]{$\subset$}}},
	Subset/.style={
		draw=none,
		every to/.append style={
			edge node={node [sloped, allow upside down, auto=false]{$\subset$}}}
	}
}
\tikzset{
	labl/.style={anchor=south, rotate=90, inner sep=.50mm}
}

\newcommand{\cinf}{C^{\infty}(M)}

\newcommand{\partderf}[2]{\frac{\partial {#1}}{\partial {#2}}}

\newcommand{\del}[1]{\delta_{#1}}

\newcommand{\ricc}{\operatorname{Ric}}
\newcommand{\weyl}{\operatorname{W}}
\newcommand{\cott}{\operatorname{C}}
\newcommand{\bach}{\operatorname{B}}

\newcommand{\supp}{\operatorname{supp}}

\newcommand\restrict[1]{\raisebox{-.5ex}{$|$}_{#1}}

\newcommand{\KN}{\mathbin{\bigcirc\mspace{-15mu}\wedge\mspace{3mu}}}

\newcommand{\Lra}{\Leftrightarrow}

\newcommand{\set}[1]{{\left\{#1\right\}}}               
\newcommand{\pa}[1]{{\left(#1\right)}}                  
\newcommand{\sq}[1]{{\left[#1\right]}}                  

\newcommand{\abs}[1]{{\left|#1\right|}}                 

\newcommand{\ul}[1]{\underline{#1}}
\newcommand{\ol}[1]{\overline{#1}}

\renewcommand{\tilde}[1]{\widetilde{#1}}

\newcommand{\interior}[1]{\accentset{\circ}{#1}}

\newtheorem{theorem}{\textbf{Theorem}}[section]

\newtheorem{cor}[theorem]{\textbf{Corollary}}

\theoremstyle{remark}
\newtheorem{rem}[theorem]{\textbf{Remark}}

\newtheorem{case}{Case}
\newtheorem{ackn}{Acknowledgments\!}

\numberwithin{equation}{section}
\patchcmd{\subequations}
{\theparentequation\alph{equation}}
{\theparentequation.\alph{equation}}
{}{}

\begin{document}
	
\title{Some canonical metrics {\em via} Aubin's local deformations}
	
	
	
	\author[1]{Giovanni Catino}
\ead{giovanni.catino@polimi.it}

\author[2]{Davide Dameno}
\ead{davide.dameno@unimi.it}

\author[3]{Paolo Mastrolia\corref{cor1}}
\ead{paolo.mastrolia@unimi.it}

\cortext[cor1]{Corresponding author}

\affiliation[1]{organization={Dipartimento di Matematica, Politecnico di Milano},
addressline={Piazza Leonardo da Vinci 32},
postcode={20133},
city={Milano},
country={Italy}}

\affiliation[2]{organization={Dipartimento di Matematica, Universita' degli Studi di Milano},
addressline={via Saldini 50},
postcode={20133},
city={Milano},
country={Italy}}

\affiliation[3]{organization={Dipartimento di Matematica, Universita' degli Studi di Milano},
addressline={via Saldini 50},
postcode={20133},
city={Milano},
country={Italy}}
	
	\begin{abstract} English: In this paper, using special metric deformations introduced by Aubin, we construct Riemannian metrics satisfying non-vanishing conditions concerning the Weyl tensor, on every compact manifold. In particular, in dimension four, we show that there are no topological obstructions for the existence of metrics with non-vanishing Bach tensor. 

\vspace{1cm}
\noindent French: Dans cet article, en utilisant des d\'eformations m\'etriques sp\'eciales introduites par Aubin, nous construisons des m\'etriques Riemanniennes satisfaisant des conditions de non-annulation concernant le tenseur de Weyl, sur toute vari\'et\'e compacte. En particulier, en dimension quatre, nous montrons qu'il n'y a pas d'obstructions topologiques \`a l'existence de m\'etriques avec un tenseur de Bach non nul.
	\end{abstract}

\vspace{1cm}

\maketitle
	\section{Introduction} \label{introd}
	
	Let $(M,g)$ be a Riemannian manifold of dimension $n\geq 3$. It is well-known
	that its Riemann curvature tensor, $\operatorname{Riem}_g$, admits the
	decomposition
	\[
	\operatorname{Riem}_g=\weyl_g+\dfrac{1}{n-2}\ricc_g\KN g-
	\dfrac{S_g}{2(n-1)(n-2)}g\KN g,
	\]
	where $\weyl_g$, $\ricc_g$, $S_g$ are the Weyl tensor, the Ricci
	tensor and the scalar curvature of $(M,g)$, respectively, and $\KN$ denotes the
	Kulkarni-Nomizu product.
	
	If we require that the curvature of $(M,g)$ satisfies certain conditions,
	several obstructions to the validity of these properties may occur:
	indeed, the topology of $M$ may not allow
	the existence of such metrics. Famous examples of this relation between curvature
	and topology are given, for instance, by metrics with positive scalar
	curvature (\cite{gromlaw2}, \cite{gromlaw}, \cite{lebrun2}, \cite{schoenyau})
	or by locally conformally flat metrics, which, for $n\geq 4$, are the ones
	with vanishing Weyl tensor (\cite{avez}, \cite{carron}, \cite{gursky},
	\cite{kuiper}).
	
	On the contrary, there are curvature conditions which can be realized
	on every Riemannian manifold (and we say that they are ``non-obstructed''): for instance, Aubin (\cite{aubin70}) showed that,
	if $M$ is closed
	and $n\geq 3$, there always
	exists a Riemannian metric $g$ such that $S_g\equiv -1$;
	he also proved that, if $M$ is compact and $n\geq 4$,
	there always exists a Riemannian metric $g$ such that the Weyl tensor
	$\weyl_g$ nowhere vanishes
	(\cite{aubin66}, \cite{aubin70}).
	The
	first author generalized these results showing that, given a Riemannian
	manifold $(M,g)$, for every $t\in\RR$,
	there exists a Riemannian metric $\tilde{g}$ such that
	the \emph{scalar-Weyl curvature} $S_g+t\abs{\weyl_g}_g\equiv-1$
	on $M$ (\cite{catino}); on the other hand, the first and the third authors,
	together with D. D. Monticelli and F. Punzo, used Aubin's result concerning the Weyl tensor to show the existence
	of \emph{weak harmonic-Weyl} metrics on every closed Riemannian four-manifold
	(\cite{catmasweyl}). More precisely, these metrics arise as minimizers of the functional
	$$
	g\longmapsto\mathfrak{D}(g):=\operatorname{Vol}_g(M)^{\frac{1}{2}} \int_M |\delta_{g} W_g|_g^2\, dV_g \,
	$$
	in the conformal class with non-vanishing Weyl tensor constructed by Aubin.

	Our main task in this paper is to investigate other curvature conditions
	which can be imposed without any topological obstruction: in particular,
	we focus on some properties involving geometric tensors related to $\weyl_g$ on compact manifolds of dimension $n\geq 4$.
	
	First, for the sake of completeness, we provide a detailed proof of Aubin's
	result (see Theorem \ref{t-aub}).
	Then, we focus on the case $n=4$: it is well-known that, on an
	oriented four-dimensional Riemannian manifold $(M,g)$, the Hodge operator $\star$ induces a
	splitting of the bundle
	of $2$-forms into two subbundles
	$\Lambda=\Lambda_+\oplus\Lambda_-,$ where $\Lambda_{\pm}$ is the eigenspace of $\star$ corresponding to
	the eigenvalue $\pm 1$. This leads to a decomposition of the Weyl tensor
	into a \emph{self-dual} and an \emph{anti-self-dual} part; namely,
	\[
	\operatorname{W_g}=\weyl_g^++\weyl_g^-.
	\]
	Exploiting Aubin's deformation method, we are able to prove the following
	
	\begin{theorem} \label{t-sd}
		Let $M$ be a compact smooth manifold, with
		$\operatorname{dim}M=4$. Then, there exists a Riemannian metric $\bar{g}$ such that
		\[
		|\weyl_{\bar{g}}^{+}|_{\bar{g}}^2\equiv 1\quad\text{on } M.
		\]
		The same
		result holds for the anti-self-dual component
		$\weyl_{\bar{g}}^-$.
	\end{theorem}
	
	As a consequence, using the metric $g_0$ constructed in Theorem \ref{t-sd} and
	following the same strategy as in \cite{catmasweyl}, it is immediate to prove the
	\begin{cor}
		On every smooth, closed four-manifold $M$, there exists a Riemannian metric $g_0$
		such that, in its conformal class $[g_0]$, there exist
		{\em weak half harmonic Weyl metrics}, i.e. minimizers of the
		quadratic curvature functional
		$$
		g\longmapsto\mathfrak{D^{\pm}}(g):=\operatorname{Vol}_g(M)^{\frac{1}{2}} \int_M |\delta_{g} W_g^{\pm}|_g^2 \, dV_g .
		$$
	\end{cor}
	\noindent
	(see also Remark $4$ in \cite{catmasweyl}).
	
	Moreover, we generalize this statement, showing a "mixed-type" condition:
	
	\begin{theorem} \label{t-mixed}
		Let $(M,g)$ be a compact Riemannian manifold, with
		$\operatorname{dim}M=4$. Then, for every $t\in\RR$, there exists
		a Riemannian metric $\bar{g}_t$ such that
		\[
		|\weyl_{\tilde{g}_t}^+ +
		t\weyl_{\tilde{g}_t}^-|^2\equiv1\quad\text{on } M.
		\]
	\end{theorem}
	
	In the subsequent sections, we focus on two other relevant geometric tensors:
	the \emph{Cotton tensor} and the \emph{Bach tensor}, which we  denote
	as $\cott_g$ and $\bach_g$, respectively (see
	Subsection \ref{prelim} for the definitions and the main properties of
	these tensors).
	
	First, we obtain a "non-obstructed" condition for $\cott_g$
	on a compact Riemannian manifold of dimension $n\geq 4$:
	
	\begin{theorem}\label{t-cot}
		Let $M$ be a compact smooth manifold of dimension $n\geq 4$.
		Then, there exists
		a metric $\tilde{g}$ such that the Cotton tensor
		$\cott_{\tilde{g}}$ of
		$(M,\tilde{g})$ vanishes only at finitely many points
		$p_1,...,p_k\in M$.
	\end{theorem}
	
	\begin{rem} \label{cottonsharp}
		We point out that Aubin's method in the proof of Theorem \ref{t-cot}
		does not lead to a sharp conclusion: indeed, one can prove the existence
		of left-invariant, non-Einstein metrics on the standard sphere whose
		Cotton tensor nowhere vanishes for every $n\geq 3$.
		Moreover, if $n=3$, the method used in the proof does not work, due to the lack of
		independent equations in the case $p\in B_{r/2}\setminus\{p_0\}$.
	\end{rem}
	
	The final section of the paper is dedicated to the tensor $\bach_g$,
	which has many applications, for instance, in General Relativity (\cite{bach}). This tensor is especially relevant when $n=4$: indeed, in this case
	$\bach_g$ is also divergence-free and conformally covariant, i.e.,
	given a conformal change $\tilde{g}=e^{2u}g$ of $g$, the Bach tensor transforms
	as
	\[
	e^{4u}\tilde{B}_{ij}=B_{ij},
	\]
	which, in global notation, means
	\[
	e^{2u}\bach_{\tilde{g}}=\bach_g
	\]
	
	When $\bach_g\equiv 0$, we say that $(M,g)$ is
	\emph{Bach-flat}: these metrics are critical
	points of the \emph{Weyl functional}
	\[
	g\longmapsto\mathcal{W}(g):=\int_M\abs{\weyl_g}_g^2dV_g,
	\]
	which is a conformally invariant functional,
	playing an important role in the study of Einstein four-manifolds:
	indeed, Bach-flatness is a necessary condition for a metric $g$ to be
	\emph{conformally Einstein} (i.e., there exists a metric $\tilde{g}$ in
	the conformal class $[g]$ such that $(M,\tilde{g})$ is an Einstein manifold).
	We point out that, in general,
	this condition is not sufficient (see \cite{salamon}):
	however, Derdzi\'{n}ski \cite{derd} showed that Bach-flatness is a sufficient
	condition for positive definite K\"{a}hler four-manifolds and recently LeBrun
	(\cite{lebrun3}) classified Bach-flat compact K\"{a}hler complex surfaces.
	
	Although the existence of topological obstructions for Bach-flat metrics
	on Riemannian four-manifolds is an open problem, in this paper we provide an answer to the "opposite" question, i.e. if the topology of the manifold plays a role in the existence of metrics with nowhere vanishing
	Bach tensor. More precisely, we exploit Aubin's construction in the four-dimensional case to obtain the following:
	
	\begin{theorem} \label{t-bac}
		Let $M$ be a compact smooth manifold with $\operatorname{dim}M=4$.
		Then, there exists a Riemannian metric $\bar{g}$ such that
		\[
		|\bach_{\bar{g}}|_{\bar{g}}^2\equiv 1\quad\text{on } M.
		\]
	\end{theorem}

	\begin{ackn}
	The authors would like to thank Professor A. Derdzi\'{n}ski for
	the useful observations appearing in Remark \ref{cottonsharp}.
	All authors are members of the	Gruppo Nazionale per le Strutture Algebriche, Geometriche e loro Applicazioni
	(GNSAGA) of INdAM (Istituto Nazionale di Alta Matematica) and have been partially supported by {\em 2022 PRIN Project: Differential-geometric aspects of manifolds via Global Analysis (code 20225J97H5)}.
	\end{ackn}
	
	\

	\section{Aubin's deformation}
	
	\subsection{Preliminaries} \label{prelim}
	The $(1, 3)$-Riemann curvature tensor of a smooth  Riemannian manifold $(M^n,g)$ is defined by
	$$
	\mathrm{R}(X,Y)Z=\nabla_{X}\nabla_{Y}Z-\nabla_{Y}\nabla_{X}Z-\nabla_{[X,Y]}Z\,.
	$$
	Throughout the article, the Einstein convention of summing over the repeated indices will be adopted. In a local coordinate system the components of the $(1, 3)$-Riemann
	curvature tensor are given by
	$R^{l}_{ijk}\tfrac{\partial}{\partial
		x^{l}}=\mathrm{R}\big(\tfrac{\partial}{\partial
		x^{j}},\tfrac{\partial}{\partial
		x^{k}}\big)\tfrac{\partial}{\partial x^{i}}$ and we denote  by
	$\operatorname{Riem}_g$ its $(0,4)$ version with components by
	$R_{ijkl}=g_{im}R^{m}_{jkl}$. The Ricci tensor is obtained by the contraction
	$R_{ik}=g^{jl}R_{ijkl}$ and $S=g^{ik}R_{ik}$ will
	denote the scalar curvature ($g^{ij}$ are the coefficient of the inverse of the metric $g$). As recalled in the Introduction, the {\em Weyl tensor} $\weyl_g$ is
	defined by the  decomposition formula, in dimension $n\geq 3$,
	\begin{eqnarray}
		\label{Weyl}
		W_{ijkl}  & = & R_{ijkl} \, - \, \frac{1}{n-2} \, (R_{ik}g_{jl}-R_{il}g_{jk}
		+R_{jl}g_{ik}-R_{jk}g_{il})  \nonumber \\
		&&\,+\frac{S}{(n-1)(n-2)} \,
		(g_{ik}g_{jl}-g_{il}g_{jk})\, \, .
	\end{eqnarray}
	The Weyl tensor shares the algebraic symmetries of the curvature
	tensor. Moreover, as it can be easily seen by the formula above, all of its contractions with the metric are zero, i.e. $\weyl$ is totally trace-free. In dimension three, $W$ is identically zero on every Riemannian manifold, whereas, when $n\geq 4$, the vanishing of the Weyl tensor is
	a relevant condition, since it is  equivalent to the local
	conformal flatness of $(M^n,g)$. We also recall that in dimension $n=3$,  local conformal
	flatness is equivalent to the vanishing of the {\em Cotton tensor} $\cott_g$, whose local components are
	\begin{equation}\label{def_cot}
		C_{ijk} =  R_{ij,k} - R_{ik,j}  -
		\frac{1}{2(n-1)}  \big( S_k  g_{ij} -  S_j
		g_{ik} \big)=
		A_{ij,k}-A_{ik,j}\,;
	\end{equation}
	here $R_{ij,k}=\nabla_k R_{ij}$ and $S_k=\nabla_k S$ denote, respectively, the components of the covariant derivative of the Ricci tensor and of the differential of the scalar curvature, and $A_{ij,k}$ denote the components
	of the covariant derivative of the \emph{Schouten tensor}
	\[
	\operatorname{A}_g=\ricc_g-\dfrac{S_g}{2(n-1)}g;
	\]
	hence, the Cotton tensor represents the obstruction for $\operatorname{A}_g$
	to be a Codazzi tensor (i.e., $(\nabla_X\operatorname{A})Y=
	(\nabla_Y\operatorname{A})X$ for every pair of vector fields $X,Y$).
	By direct computation, we can see that $\cott_g$
	satisfies the  symmetries
	\begin{equation}\label{CottonSym}
		C_{ijk}=-C_{ikj},\,\quad\quad C_{ijk}+C_{jki}+C_{kij}=0\,,
	\end{equation}
	moreover it is totally trace-free,
	\begin{equation}\label{CottonTraces}
		g^{ij}C_{ijk}=g^{ik}C_{ijk}=g^{jk}C_{ijk}=0\,,
	\end{equation}
	by its skew--symmetry and Schur lemma. We also recall that, for $n\geq 4$,  the Cotton tensor can  be defined as one of the possible divergences of the Weyl tensor:
	\begin{equation}\label{def_Cotton_comp_Weyl}
		C_{ijk}=\pa{\frac{n-2}{n-3}}W_{tikj, t}=-\pa{\frac{n-2}{n-3}}W_{tijk, t}=-\frac{n-2}{n-3} (\delta W)_{ijk} \,.
	\end{equation}
	A computation shows that the two definitions coincide (see e.g. \cite{cmbook}).	
	
	\
	
	The
	\emph{Bach tensor} $\bach_g$ of $(M,g)$ is defined,
	in components, as
	\begin{equation} \label{bach}
		B_{ij}:=\dfrac{1}{n-2}\pa{g^{ks}C_{jik,s}+g^{ks}g^{lt}R_{kl}W_{isjt}}.
	\end{equation}
	It is immediate to show that $\bach_g$ is a traceless tensor; moreover,
	since $(n-3)W_{jkil,lk}=(n-2)C_{ijk,k}$, exploiting the second covariant
	derivative commutation formulas, it can be shown that $\bach_g$
	is symmetric (see, for instance, \cite[Lemma 2.8]{cmbook}). Also, recall that,
	if $n=4$, the Bach tensor acquires two additional features: it is
	divergence-free	and conformally covariant.
	
	\subsection{Aubin's local deformations} Let us introduce
	the following deformation of the metric $g$:
	\begin{equation} \label{aubindef}
		\tilde{g}=g+d\phi\otimes d\phi,		
	\end{equation}
	where $\phi\in C^{\infty}(M)$.
	We denote the Weyl tensor of $(M,\tilde{g})$ as
	$\weyl_{\tilde{g}}$. If $U$ is a local chart of $M$
	and $x_1,...,x_n$ are local coordinates on $U$, the local components
	of the $(0,4)$-version of $\weyl_{\tilde{g}}$,  $\tilde{W}_{ijkt}$, are given by the following expression
	(see also \cite{cmbook}, Chapter 2):
	\begin{align} \label{weylaubin}
		\tilde{W}_{ijkt}&=W_{ijkt}+\dfrac{1}{w}(\phi_{ik}\phi_{jt}-
		\phi_{it}\phi_{jk})+\\
		&+\dfrac{1}{n-2}(R_{ik}\phi_j\phi_t-R_{it}\phi_j\phi_k+
		R_{jt}\phi_i\phi_k-R_{jk}\phi_i\phi_t) \notag\\
		&+\dfrac{S}{(n-1)(n-2)}(g_{ik}\phi_j\phi_t-g_{it}\phi_j\phi_k+
		g_{jt}\phi_i\phi_k-g_{jk}\phi_i\phi_t)+\notag\\
		&+\dfrac{\phi^p\phi^q}{w(n-2)}[
		R_{ipkq}(g_{jt}+\phi_j\phi_t)-R_{iptq}(g_{jk}+\phi_j\phi_k)+
		R_{jptq}(g_{ik}+\phi_i\phi_k)-R_{jpkq}(g_{it}-\phi_i\phi_t)]+
		\notag\\
		&-\dfrac{2R_{pq}\phi^p\phi^q}{w(n-1)(n-2)}
		[g_{ik}g_{jt}-g_{it}g_{jk}+g_{ik}\phi_j\phi_t-g_{it}\phi_j\phi_k+
		g_{jt}\phi_i\phi_k-g_{jk}\phi_i\phi_t]+\notag\\
		&-\dfrac{1}{w(n-2)}\{[(\Delta\phi)\phi_{ik}-\phi_{ip}\phi_k^p]
		(g_{jt}+\phi_j\phi_t)-[(\Delta\phi)\phi_{it}-\phi_{ip}\phi_t^p]
		(g_{jk}+\phi_j\phi_k)\}+\notag\\
		&-\dfrac{1}{w(n-2)}\{[(\Delta\phi)\phi_{jt}-\phi_{jp}\phi_t^p]
		(g_{ik}+\phi_i\phi_k)-[(\Delta\phi)\phi_{jk}-\phi_{jp}\phi_k^p]
		(g_{it}+\phi_i\phi_t)\}+\notag\\
		&+\dfrac{1}{w(n-1)(n-2)}\sq{(\Delta\phi)^2-
			\abs{\operatorname{Hess}(\phi)}^2}	
		[g_{ik}g_{jt}-g_{it}g_{jk}+g_{ik}\phi_j\phi_t-g_{it}\phi_j\phi_k+
		g_{jt}\phi_i\phi_k-g_{jk}\phi_i\phi_t]+\notag\\
		&+\dfrac{\phi^p\phi^q}{w^2(n-2)}[(\phi_{ik}\phi_{pq}-\phi_{ip}
		\phi_{kq})(g_{jt}+\phi_j\phi_t)-(\phi_{it}\phi_{pq}-\phi_{ip}
		\phi_{tq})(g_{jk}+\phi_j\phi_k)]+\notag\\
		&+\dfrac{\phi^p\phi^q}{w^2(n-2)}[(\phi_{jt}\phi_{pq}-\phi_{jp}
		\phi_{tq})(g_{ik}+\phi_i\phi_k)-(\phi_{jk}\phi_{pq}-\phi_{jp}
		\phi_{kq})(g_{it}+\phi_i\phi_t)]+\notag\\
		&-\dfrac{2}{w^2(n-1)(n-2)}[(\Delta\phi)\phi^p\phi^q\phi_{pq}-
		\phi^p\phi_{pq}\phi^{qr}\phi_r](g_{ik}g_{jt}-g_{it}g_{jk})+\notag\\
		&-\dfrac{2}{w^2(n-1)(n-2)}[(\Delta\phi)\phi^p\phi^q\phi_{pq}-
		\phi^p\phi_{pq}\phi^{qr}\phi_r]
		(g_{ik}\phi_j\phi_t-g_{it}\phi_j\phi_k+
		g_{jt}\phi_i\phi_k-g_{jk}\phi_i\phi_t)\notag,
	\end{align}
	where $w=1+\abs{\nabla\phi}^2$ and
	\begin{align*}
		\phi_i&=\partial_i\phi=\partderf{\phi}{x_i},\\
		\phi^i&=g^{ip}\phi_p,\\
		\phi_{ij}&=\partial_i\partial_j\phi-\Gamma_{ij}^p\phi_p,\\
		\phi_j^i&=g^{ip}\phi_{pj}=\partial_j\phi^i+\phi^p\Gamma_{pj}^i,\\
		\phi^{ij}&=g^{ip}\phi_p^j.
	\end{align*}
	
	\
	
	\section{A detailed proof of Aubin's result}
	
	In this section we give a complete proof of Aubin's result
	(see \cite{aubin66} and
	\cite{aubin70}), i.e. we prove the following
	\begin{theorem}[Aubin (\cite{aubin66}, \cite{aubin70})] \label{t-aub}
		On every smooth manifold of dimension at least $4$ there exists
		a Riemannian metric $g$ whose Weyl tensor nowhere identically vanishes.
	\end{theorem}
	
	\begin{proof}
		We divide the proof in two steps.
		
		\vspace{0.3cm}
		
		\begin{center}
			\large
			\textbf{Step 1: the local deformation}
		\end{center}
		
		\vspace{0.3cm}
		
		Let $g$ any Riemannian metric on $M$ and consider the metric $\tilde{g}$ given
		by \eqref{aubindef}. Let $p_0\in M$ be such that
		$\weyl_g$ vanishes at $p_0$ and $B_r$ an open ball of
		radius $r$ and centered in $p_0$. Moreover, let us
		consider normal coordinates $x_1,...,x_n$ on $B_r$ such that
		$p_0=(0,...,0)$.
		Thus, at $p_0$ we have
		\[
		g_{ij}=g^{ij}=\delta_{ij}, \quad
		\phi_i=\phi^i, \quad \phi_{ij}=\partial_i\partial_j\phi=\phi_j^i=
		\phi^{ij}
		\]
		From now on, we denote the local components of
		$\weyl_g$ ($\weyl_{\tilde{g}}$, resp.) on $B_r$ as
		$W_{ijkl}$ ($\tilde{W}_{ijkl}$, resp.).
		
		We construct the function $\phi$ as follows: let
		$f\in C^{\infty}([0,+\infty))$ such that
		\[
		\begin{cases}
			f(y)=0, \mbox{ if } y\geq 1\\
			f'(y)>0, f''(y)<0, \mbox{ if } 0\leq y< 1
		\end{cases}.
		\]
		For instance, we may choose
		\begin{equation} \label{rightfunction}
			f(x):=
			\begin{cases}
				-e^{\pa{\frac{b}{1-x}}} &\mbox{ if } 0\leq x<1\\
				0 &\mbox{ if } x\geq 1
			\end{cases},
		\end{equation}
		where $b>0$ is sufficiently large.
		Now, let $\lambda, \alpha_1,...,\alpha_n$ be $n+1$ real numbers in the
		interval $[1,2]$ and let
		\begin{equation} \label{phiweyl}
			\phi=\dfrac{\lambda r^2}{2}f
			\pa{\dfrac{\alpha_1x_1^2+...+\alpha_nx_n^2}{r^2}}.
		\end{equation}
		By definition, $\phi\in C^{\infty}(B_r)$ and
		\[
		B_{\frac{r}{2}}\subset\operatorname{supp}\phi\subset B_r.
		\]
		Indeed, if $x_1,...,x_n$ are such that
		$\alpha_1x_1^2+...+\alpha_nx_n^2<r^2$, then, since $\alpha_i\geq 1$
		for every $i$,
		\[
		\sum_{i=1}^n x_i^2\leq\sum_{i=1}^n\alpha_ix_i^2 < r^2,
		\]
		i.e. $p=(x_1,...,x_n)\in B_r$;
		on the other hand, if $p\in B_{\frac{r}{2}}$, then, since
		$\alpha_i\leq 2$ for every $i$,
		\[
		\sum_{i=1}^n\alpha_ix_i^2\leq 2\sum_{i=1}^nx_i^2 < \dfrac{r^2}{2},
		\]
		thus $\alpha_1x_1^2+...+\alpha_nx_n^2<r^2$ and
		$p\in\operatorname{supp}\phi$.
		
		The partial derivatives of $\phi$ satisfy
		\begin{equation} \label{derphiweyl}
			\phi_i=\lambda f'\cdot \alpha_ix_i=
			O(r),
		\end{equation}
		as $r\to 0$. From now on, every $O(\cdot)$ will be regarded
		as $r\to 0$. Since we chose a system of
		normal coordinates, for small radii the second partial derivatives
		of $\phi$ satisfy
		\begin{equation} \label{2derphiweyl}
			\phi_{ij}=\lambda\pa{\alpha_if'\del{ij}+
				2\dfrac{\alpha_i\alpha_j}{r^2}x_ix_jf''}=O(1).
		\end{equation}
		
		Now, let us consider equation \eqref{weylaubin}:
		we can rewrite the expression as
		\begin{align} \label{apprweyl}
			\tilde{W}_{ijkl}&=W_{ijkl}+\phi_{ik}\phi_{jl}
			-\phi_{il}\phi_{jk}+\\
			&-\dfrac{1}{n-2}\Delta\phi(\phi_{ik}\del{jl}-\phi_{il}\del{jk}+
			\phi_{jl}\del{ik}-\phi_{jk}\del{il})+\notag\\
			&+\dfrac{1}{n-2}(\phi_{ip}\phi_{pk}\del{jl}-\phi_{ip}
			\phi_{pl}\del{jk}+
			\phi_{jp}\phi_{pl}\del{ik}-\phi_{jp}\phi_{pk}\del{il})\notag\\
			&+\dfrac{1}{(n-1)(n-2)}\sq{(\Delta\phi)^2-
				\abs{\operatorname{Hess}(\phi)}^2}(\del{ik}\del{jl}-
			\del{il}\del{jk})+
			O(r^2).\notag
		\end{align}
		Thus, we informally
		distinguish a ``principal part'' and a ``remainder'' in the
		expression of the components $\tilde{W}_{ijkl}$.
		We define
		\begin{equation} \label{support}
		S:=\interior{\supp\phi}=
		\left\{p=(x_1,...,x_n)\in B_r: \sum_{i=1}^n
		\alpha_ix_i^2<r^2\right\};
		\end{equation}
		the key of
		the proof is to show that the principal parts of the components
		$\tilde{W}_{ijkl}$ cannot be simultaneously zero on $S$.
		
		Now, let
		$i\neq j\neq k\neq l$; inserting \eqref{derphiweyl} and
		\eqref{2derphiweyl} into \eqref{apprweyl}, we obtain
		\begin{align} \label{weylremainder}
			\tilde{W}_{ijij}&=W_{ijij}+\lambda^2\sq{a_{ij}(f')^2+
				b_{ij}f'f''}+O(r^2);\\
			\tilde{W}_{ijik}&=W_{ijik}+\lambda^2a_{ijk}f'f''x_jx_k
			+O(r^2);\notag\\
			\tilde{W}_{ijkl}&=W_{ijkl}+O(r^2),\notag
		\end{align}
		where
		\begin{align} \label{coeffweyl}
			a_{ij}&=\dfrac{1}{n-2}\sq{(n-4)\alpha_i\alpha_j-
				(\alpha_i+\alpha_j)\sum_{k\neq i,j}\alpha_k+
				\dfrac{2}{n-1}\sum_{k<l}\alpha_k\alpha_l};\\
			b_{ij}&=\dfrac{2}{(n-2)r^2}\left[
			(n-4)(\alpha_ix_i^2+\alpha_jx_j^2)\alpha_i\alpha_j-
			(\alpha_i^2x_i^2+\alpha_j^2x_j^2)\sum_{k\neq i,j}\alpha_k+
			\right.\notag\\
			&\left.-(\alpha_i+\alpha_j)\sum_{k\neq i,j}\alpha_k^2x_k^2
			+\dfrac{2}{n-1}\sum_{k=1}^n\alpha_k
			\pa{\sum_{l\neq k}\alpha_l^2x_l^2}\right];\notag\\
			a_{ijk}&=\dfrac{2\alpha_j\alpha_k}{(n-2)r^2}
			\sq{(n-3)\alpha_i-\sum_{l\neq i,j,k}\alpha_l}.\notag
		\end{align}
		Note that $a_{ij}\in\RR$, $b_{ij}=b_{ij}(r,p)$ and
		$a_{ijk}=a_{ijk}(r)$, but $a_{ij}$, $b_{ij}$
		and $a_{ijk}x_jx_k$ are
		$O(1)$, for every $i,j,k$.
		It is important to note that there exist suitable choices for
		$\alpha_1,...,\alpha_n$ such that, for every $i\neq j\neq k$,
		$a_{ij}$ and $a_{ijk}$ nowhere vanish on $S$
		(observe that $a_{ij}$ and $a_{ijk}$ are scalars, while
		$b_{ij}$ is a polynomial of degree $2$ in the variables
		$x_1,...,x_n$ for every $i\neq j\neq k$). For instance, we
		may define
		\[
		\begin{cases}
			(\alpha_1,...,\alpha_n)=(2,2,1,1,...,1), \mbox{ if } n>4;\\
			(\alpha_1,\alpha_2,\alpha_3,\alpha_4)=
			\pa{1,\frac{5}{4},\frac{3}{2},2},
			\mbox{ if } n=4.
		\end{cases}
		\]
		A direct inspection of \eqref{coeffweyl} shows
		that, with this choice, $a_{ij},a_{ijk}\neq 0$.
		
		Note that, for $n=4$, $\alpha_i\neq\alpha_j$ if
		$i\neq j$. For $n>4$, observe that $a_{ij}$ and $a_{ijk}$
		can be seen as
		homogeneous polynomials in the $n$ variables
		$\alpha_1,...,\alpha_n$, therefore, in particular, they
		are smooth functions of these variables: hence, since we found
		a $n$-tuple $(\alpha_1,...,\alpha_n)$ such that
		$a_{ij},a_{ijk}\neq 0$, we know that there exist sufficiently
		small
		$\epsilon_1\neq...\neq\epsilon_n$, with $\epsilon_i>0$ for every
		$i$, such that $a_{ij},a_{ijk}\neq 0$ for
		\[
		(\alpha_1',...,\alpha_n'):=(2-\epsilon_1,2-\epsilon_2,
		1+\epsilon_3,1+\epsilon_4,...,1+\epsilon_n)
		\]
		and $\alpha_i'\neq\alpha_j'$ for $i\neq j$.
		Therefore, without loss of generality, we may assume
		that $\alpha_i\neq\alpha_j$ whenever $i\neq j$.
		
%
%
		
		Let us distinguish three cases.
		\begin{case}[$p=p_0$] By hypothesis, $\weyl_g$ vanishes
			at $p$ and, since $p_0=(0,...,0)$ in our local coordinates,
			by \eqref{weylremainder} we obtain
			\begin{align*}
				\tilde{W}_{ijij}&=\lambda^2a_{ij}(f')^2+O(r^2);\\
				\tilde{W}_{ijik}&=O(r^2);
			\end{align*}
			since $a_{ij}, f', \lambda\neq 0$, we have that
			\[
			\abs{\weyl_{\tilde{g}}}^2_{\tilde{g}}
			\geq 2\sum_{i<j}\tilde{W}_{ijij}^2=
			(\lambda f')^4\sum_{i<j}(a_{ij})^2>0.
			\]
		\end{case}
		\begin{case}[$p\in B_{r/2}\setminus \{p_0\}$] We want
			to show that the components of the Weyl tensor
			$\operatorname{W}_{\tilde{g}}$ cannot vanish
			simultaneously at $p$, if $r$ is sufficiently small, i.e.
			$r<\bar{r}=\bar{r}(p_0,||g||_{C^k})$, for $k\geq 3$.
			Since $p$ lies in
			the open ball of radius $r/2$ and centered in $p_0$, by
			Taylor's Theorem we have that
			\[
			\abs{\weyl_g}\leq C\cdot r + O(r^2).
			\]
			Let us suppose $\tilde{W}_{ijij}=\tilde{W}_{ijik}=0$
			for every $i\neq j\neq k$. By \eqref{weylremainder},
			we can write
			\begin{align*}
				a_{ij}(f')^2+b_{ij}f'f''+O(r)&=0;\\
				a_{ijk}x_jx_k+O(r)&=0.
			\end{align*}
			For a sufficiently small radius $r$,
			the previous equations imply
			\begin{subequations}
				\begin{empheq}[left=\empheqlbrace]{align}
					a_{ij}(f')^2+b_{ij}f'f''&=0; \label{wijij}\\
					a_{ijk}x_jx_k&=0. \label{wijik}
				\end{empheq}
			\end{subequations}
			Note that we obtained an overdetermined system in the variables
			$x_1,...x_n$: indeed, since $i\neq j\neq k$ and the coefficients
			$a_{ijk}$ are symmetric with respect to the indices $j$ and $k$,
			we have $n(n-1)/2$ independent
			equations of the form $\eqref{wijik}$ (observe that
			changing the index $i$ in \eqref{wijik} does not
			provide additional equations).
			Moreover, the polynomials $a_{ij}(f')^2+b_{ij}f'f''$ are
			symmetric with respect to $i$ and $j$ and a straightforward
			computation shows that
			\[
			\sum_{i\neq j}a_{ij}=\sum_{i\neq j}b_{ij}=0, \mbox{ for every }
			j
			\]
			(this can also be seen as a consequence of the fact that
			the Weyl tensor is traceless). Thus, we have
			\[
			\dfrac{n(n-1)}{2}-n=\dfrac{n(n-3)}{2}
			\]
			equations of the form \eqref{wijij}. Therefore, our system
			is made by
			\[
			\dfrac{n(n-3)}{2}+\dfrac{n(n-1)}{2}=n(n-2)
			\]
			independent equations, and $n(n-2)>n+1>n$ for every $n\geq 4$.
			
			Now, let us show that the system admits only the solution
			$x_1=\dots=x_n=0$, which will lead to a contradiction,
			since $p\neq p_0$.
			Since $a_{ijk}\neq 0$, we obtain that $x_jx_k=0$ for every
			$j\neq k$. This implies that at least $n-1$ coordinates
			of $p$ must be zero; since $p\neq p_0$,
			there is exactly one coordinate $x_i$ which is non-zero.
			
			Let us consider $j\neq t\neq s\neq i$ (note that
			this is possible since $n\geq 4$):
			by $\tilde{W}_{ijij}=\tilde{W}_{itit}=\tilde{W}_{isis}=0$ we obtain
			\begin{align*}
				0&=\dfrac{1}{n-2}\sq{(n-4)\alpha_i\alpha_j-
					(\alpha_i+\alpha_j)\sum_{k\neq i,j}\alpha_k+
					\dfrac{2}{n-1}\sum_{k<l}\alpha_k\alpha_l}(f')^2+\\
				&+\dfrac{2}{(n-2)r^2}\left[
				(n-4)\alpha_i^2\alpha_jx_i^2-
				\alpha_i^2x_i^2\sum_{k\neq i,j}\alpha_k
				+\dfrac{2}{n-1}\sum_{k=1}^n\alpha_k
				\pa{\sum_{l\neq k}\alpha_l^2x_l^2}\right]f'f'';\\
				0&=\dfrac{1}{n-2}\sq{(n-4)\alpha_i\alpha_t-
					(\alpha_i+\alpha_t)\sum_{k\neq i,t}\alpha_k+
					\dfrac{2}{n-1}\sum_{k<l}\alpha_k\alpha_l}(f')^2+\\
				&+\dfrac{2}{(n-2)r^2}\left[
				(n-4)\alpha_i^2\alpha_tx_i^2-
				\alpha_i^2x_i^2\sum_{k\neq i,t}\alpha_k
				+\dfrac{2}{n-1}\sum_{k=1}^n\alpha_k
				\pa{\sum_{l\neq k}\alpha_l^2x_l^2}\right]f'f''; \\
				0&=\dfrac{1}{n-2}\sq{(n-4)\alpha_i\alpha_s-
					(\alpha_i+\alpha_s)\sum_{k\neq i,s}\alpha_k+
					\dfrac{2}{n-1}\sum_{k<l}\alpha_k\alpha_l}(f')^2+\\
				&+\dfrac{2}{(n-2)r^2}\left[
				(n-4)\alpha_i^2\alpha_sx_i^2-
				\alpha_i^2x_i^2\sum_{k\neq i,s}\alpha_k
				+\dfrac{2}{n-1}\sum_{k=1}^n\alpha_k
				\pa{\sum_{l\neq k}\alpha_l^2x_l^2}\right]f'f'';
			\end{align*}
			subtracting the second and the third equations from the first,
			since $\alpha_j\neq\alpha_t\neq\alpha_s$ and
			$f',f''\neq 0$ on $S$, we get
			\begin{align*}
				0&=\sq{(n-3)\alpha_i-
					\sum_{k\neq i,j,t}\alpha_k}f'+
				\dfrac{2}{r^2}(n-3)\alpha_i^2x_i^2f'',\\
				0&=\sq{(n-3)\alpha_i-
					\sum_{k\neq i,j,s}\alpha_k}f'+
				\dfrac{2}{r^2}(n-3)\alpha_i^2x_i^2f''.
			\end{align*}
			It is immediate to observe that these two equations hold
			simultaneously if and only if
			\[
			\sum_{k\neq i,j,t}\alpha_k=\sum_{k\neq i,j,s}\alpha_k \quad
			\Lra \quad \alpha_s=\alpha_t,
			\]
			which is impossible.
			Thus, not all the components of $\weyl_{\tilde{g}}$
			vanish at $p$.
		\end{case}
		\begin{case}[$p\in S\setminus B_{\frac{r}{2}}$]
			Let us suppose again
			that $\tilde{W}_{ijij}=\tilde{W}_{ijik}=0$ for every
			$i\neq j\neq k$. As in Case 2, for a sufficiently small
			$r$, the first two equations
			in \eqref{weylremainder} imply
			\begin{subequations}
				\begin{empheq}[left=\empheqlbrace]{align}
					W_{ijij}&+\lambda^2(a_{ij}(f')^2+b_{ij}f'f'')=0;\\
					W_{ijik}&+\lambda^2a_{ijk}x_jx_kf'f''=0.
				\end{empheq}
			\end{subequations}
			If $W_{ijij}=W_{ijik}=0$ at $p$, we get a contradiction
			by the conclusions of Case 2. Thus, let us suppose that
			$\abs{\weyl_g}_g^2>0$ at $p$: for instance, let $W_{ijik}\neq 0$
			for some $i,j,k$. The equation $\tilde{W}_{ijik}=0$ allows us
			to compute $\lambda$:
			\[
			\lambda^2=-\dfrac{W_{ijik}}{a_{ijk}x_jx_k}.
			\]
			This equation holds for every point whose coordinates are solutions
			of the system above; however, $\lambda\in [1,2]$ appears as
			a free parameter in \eqref{phiweyl}, therefore it is sufficient
			to choose $\lambda_1\in [1,2]$ such that $\lambda_1^2\neq\lambda^2$
			and repeat the argument of the proof to obtain a contradiction.
			Thus, $W_{ijik}=0$. If, for instance, $\lambda_1$ is such that the
			equation
			\[
			W_{i'j'i'k'}+\lambda_1^2a_{i'j'k'}x_{j'}x_{k'}f'f''=0
			\]
			holds for some $i'\neq j'\neq k'$, it is sufficient to choose
			$\lambda_2\in [1,2]$ such that $\lambda_2^2\neq\lambda_1^2$ to
			get the same contradiction. Note that we can repeat the procedure
			for every equation of the system above.
			
			Therefore, possibly choosing $\lambda$ in \eqref{phiweyl} out
			of a finite set $\{\lambda_1,...,\lambda_k\}$, we can conclude
			that the system holds if and only if $W_{ijij}=W_{ijik}=0$ at $p$:
			however, by the argument of Case 2, this leads to a contradiction.
		\end{case}
	
		\vspace{2cm}
		
		\begin{center}
			\large
			\textbf{Step 2: iteration of the process}
		\end{center}
		
		\vspace{0.3cm}
		
		In the first step,
		we proved that the Weyl tensor $\weyl_{\tilde{g}}$ does not
		vanish on $S$. Now, let us call $g_0=g$, $\phi^0=\phi$,
		$S_0=S$, $r_0=r$, $\lambda_0=\lambda$ and
		$g_1=\tilde{g}$: given $p_0\in M$ such that
		$\abs{\weyl_{g_0}}_{g_0}(p_0)=0$, there exist a normal open
		neighborhood $U_0$ and $\phi^0\in\cinf$, defined as in
		\eqref{phiweyl} with $r_0$ and $\lambda_0$,
		such that $S_0=\interior{\supp\phi^0}\subset U_0$
		and $\weyl_{g_1}$ has non-vanishing square norm on $S_0$, where
		$g_1=g_0+d\phi^0\otimes d\phi^0$.
		Since $M$ is compact by hypothesis, the set
		\[
		Z:=\set{p\in M: \abs{\weyl_{g_0}}_{g_0}(p)=0}
		\]
		is compact: indeed, $Z$ is closed, since it is the zero locus
		of a continuous function on $M$.
		Therefore, there exists a finite open cover of $Z$
		of the form
		\[
		\bigcup_{i=1}^NV_i:=\bigcup_{i=1}^N(S_i\cap Z),
		\]
		where $S_i$ contains a point $p_i$ where $\weyl_{g_0}$ vanishes
		and it is the interior of the support of a smooth function
		$\phi^i$ defined as in \eqref{phiweyl}, with
		$r_i$ small enough and $\lambda_i$ such that Aubin's local deformation
		can be performed as before. Moreover, observe that, if
		$p_j\in Z$, then, by construction, $p_j\not\in V_k$
		if $j\neq k$; we also note that Aubin's deformation on $S_i$
		do not produce new zeroes of $\abs{\weyl_{g_0}}$ outside of
		$Z$, which means that, if $p'\not\in Z$ before
		the deformation, then $p'\not\in Z$ after deforming
		the metric as well.
		
		The first step of the proof was to show that, around $p_0$,
		the metric $g_0$ can be deformed in order to have $\weyl_{g_1}
		\not\equiv 0$ on $S_0$. Now, we perform the argument again:
		let $p_1$ such that $\weyl_{g_0}\equiv 0$ at $p_1$
		and let $V_1\ni p_1$, with deformation function $\phi^1$,
		which has $\lambda_1$ and $r_1$ in its definition (recall
		that $r_1$ is chosen small enough so that Aubin's method
		can be exploited).
		
		If $V_0\cap V_1=\emptyset$, we can apply the deformation in $S_1$ in order to conclude
		that $\weyl_{g_2}\not\equiv 0$ on $S_1$, where
		$g_2=g_0+d\phi^1\otimes d\phi^1$, and, hence, on $V_1$.
		Therefore, let us suppose that $V_0\cap V_1\neq\emptyset$:
		if we consider
		a point $p\in V_1\setminus V_0$, here $g_1=g_0$, hence
		we Aubin's argument on $S_1$ works as in the previous case.
		Let us suppose that
		there exists a point $q\in V_0\cap V_1$ such that
		$\weyl_{g_2}$ vanishes identically at $q$: in this case,
		we have
		\[
		g_2=g_1+d\phi^1\otimes d\phi^1.
		\]
		The expression for the components of $\weyl_{g_2}$ is given by
		\eqref{weylaubin}, where $g_{ij}=(g_1)_{ij}$ and both the
		covariant derivatives of $\phi=\phi_1$ and the curvature quantities
		are referred to the metric $g_1$.
		
		Let us choose the
		indices $i,j,k,t$ such that $W^1_{ijkt}\neq 0$ at $q$ (whose
		existence is guaranteed by the first deformation we performed).
		If we evaluate \eqref{weylaubin} at $q$, the left-hand side
		vanishes: hence, since $\alpha_1,...,\alpha_n$ and $r_1$ are fixed,
		if we multiply both sides by $w^2$ we obtain an equation of
		the form
		\begin{equation} \label{nonhomequ}
			0=W^1_{ijkt}+P_{ijkt}(\lambda_1),
		\end{equation}
		where $W^1_{ijkt}=W^1_{ijkt}(q)$ and
		$P_{ijkt}(\lambda)=\sum_{i=1}^MC_i(\lambda_1)^i$ is a non-trivial
		polynomial of degree $M$ in $\lambda_1$. Thus,
		\eqref{nonhomequ} is a non-homogeneous polynomial equation
		in $\lambda_1$ with real coefficients, which means that
		the set of its roots is
		\[
		L_1=\{(\lambda_1)_1,...,(\lambda_1)_K\}, \quad K\leq M.
		\]
		Note that, if $\lambda_1=(\lambda_1)_{K'}$,
		for some $1\leq K'\leq K$, since $\lambda_1$ is a real number,
		then every other point $q'$ such that $\abs{\weyl_{g_2}}_{g_2}(q')=0$
		must satisfy \eqref{nonhomequ} with $\lambda_1=(\lambda_1)_{K'}$.
		
		Since the set of values of $\lambda_1$ such that $\weyl_{g_2}$
		vanishes at $q$ is finite, it is sufficient to choose
		$\lambda_1=\bar{\lambda}_1$ in $[1,2]\setminus L_1$
		to get a contradiction: therefore, up to choose $\lambda_1$ outside
		of a finite set of values, we have that $\weyl_{g_2}$ does not
		vanish in $V_0\cap V_1$, which implies that
		\[
		\abs{\weyl_{g_2}}_{g_2}\neq 0 \mbox{ on } V_0\cup V_1.
		\]
		Since $\{V_0,...,V_N\}$ is a finite set, we have a finite number
		of non-empty intersections: hence, we can repeat the process finitely many
		times to conclude that there exists a metric $\tilde{g}$ such that
		$\weyl_{\tilde{g}}\not\equiv 0$ on $M$ and this ends the proof.
	\end{proof}
	\setcounter{case}{0}
	\begin{rem} \label{weyl1}
		If $\abs{\weyl_{\tilde{g}}}_{\tilde{g}}>0$ for every point of $M$,
		then, operating the conformal change
		\[
		\ol{g}:=\abs{\weyl_{\tilde{g}}}\tilde{g},
		\]
		we obtain that the metric $\ol{g}$ is such that its Weyl tensor
		$\weyl_{\ol{g}}$ satisfies
		\[
		\abs{\weyl_{\ol{g}}}_{\ol{g}}^2\equiv 1 \mbox{ on } M.
		\]
	\end{rem}
	
	\
	
	\section{Proof of Theorems \ref{t-sd} and \ref{t-mixed}}
	In this section we extend Aubin's result in dimension four to the
	self-dual and anti-self dual components of the Weyl tensor in order to
	prove Theorem \ref{t-sd}.
	\begin{proof}[Proof of Theorem \ref{t-sd}]
		First, note that, by Remark \ref{weyl1}, it is sufficient to show
		that there exists a Riemannian metric whose self-dual Weyl tensor
		nowhere vanishes on $M$.
		
		Similarly as we did in the proof of Theorem \ref{t-aub},
		let $g$ any Riemannian metric on $M$ and let again $p_0\in M$ be
		such that $\operatorname{W^+}(p_0)=0$. We choose an open ball
		$B_r$ centered at $p_0$ with normal coordinates
		$x_1,x_2,x_3,x_4$ such that $p_0=(0,0,0,0)$ and we define
		a function $\phi$ as in \eqref{phiweyl} in such a way
		that $B_{\frac{r}{2}}\subset\interior{\supp}\phi\subset B_r$. Let
		$S=\interior{\supp}\phi$ and
		$\tilde{g}$ be the metric defined in \eqref{aubindef}.
		
		By
		definition
		\[
		W_{ijkl}=W_{ijkl}^+ + W_{ijkl}^-;
		\]
		moreover, it is not hard to show that, for every $i,j,k,l=1,...,4$
		such that $i\neq j$ and $k\neq l$, there exist indices $k'$ and $l'$
		such that
		\[
		W_{ijkl}^{\pm}=\pm W_{ijk'l'}^{\pm}.
		\]
		The pair $(k',l')$ is uniquely determined by the action of
		the Hodge star operator $\star$: indeed, it is well-known that
		the terms in which $\weyl$ decomposes are given by
		\[
		W_{ijkl}^{\pm}=\dfrac{1}{2}\sq{W_{ijkl}\pm(\star W)_{ijkl}}
		\]
		(for a detailed discussion, see, for instance, \cite{besse,singthor}).
		This implies immediately that
		\[
		W_{ijkl}^{\pm}=\dfrac{1}{2}(W_{ijkl}\pm W_{ijk'l'}).
		\]
		Let us now focus on $\weyl_g^+$. By \eqref{weylremainder}
		and \eqref{coeffweyl},
		for $i\neq j$ one can easily obtain
		\begin{align} \label{wijijplus}
			\tilde{W}_{ijij}^+&=\dfrac{1}{2}
			(\tilde{W}_{ijij}+\tilde{W}_{iji'j'})=\\
			&=\dfrac{1}{2}[W_{ijij}+W_{iji'j'}+\lambda^2(
			a_{ij}(f')^2+b_{ij}f'f'')+O(r^2)]=\notag \\
			&=W_{ijij}^+ +\dfrac{\lambda^2}{2}(
			a_{ij}(f')^2+b_{ij}f'f'') + O(r^2) \notag
		\end{align}
		(note that $(i',j')=(k,l)$ are such that $i\neq j\neq k\neq l$).
		Analogously, for $i\neq j\neq k$, we obtain
		\begin{align} \label{wijikplus}
			\tilde{W}_{ijik}^+&=\dfrac{1}{2}
			(\tilde{W}_{ijik}+\tilde{W}_{iji'k'})=\\
			&=\dfrac{1}{2}[W_{ijik}\pm W_{jijl}+\lambda^2(
			a_{ijk}x_jx_k\pm a_{jil}x_ix_l)f'f''+O(r^2)]=\notag\\
			&=W_{ijik}^+ +\dfrac{\lambda^2}{2}(
			a_{ijk}x_jx_k\pm a_{jil}x_ix_l)f'f''+O(r^2). \notag
		\end{align}
		Here, $\pm$ appears in the equations since we may have
		$(i',k')=(l,j)$ or $(i',k')=(j,l)$.
		
		Now, we are ready to prove the statement. Let us choose
		\[
		(\alpha_1,\alpha_2,\alpha_3,\alpha_4)=
		\pa{1,\frac{5}{4},\frac{3}{2},2};
		\]
		thus, an easy computation shows that
		\begin{equation} \label{coeffweylplus}
			\begin{cases}
				a_{12}=\frac{5}{48}=a_{34}\\
				a_{13}=-\frac{1}{48}=a_{24}\\
				a_{14}=-\frac{1}{12}=a_{23}
			\end{cases}
			\mbox{ and }
			\begin{cases}
				a_{123}=-\frac{15}{8r^2}, \mbox{ } a_{214}=-\frac{1}{2r^2}\\
				a_{124}=-\frac{5}{4r^2}, \mbox{ } a_{213}=-\frac{9}{8r^2}\\
				a_{134}=-\frac{3}{4r^2}, \mbox{ } a_{312}=-\frac{5}{8r^2}\\
			\end{cases}.
		\end{equation}
		We recall that
		\[
		\sum_{i\neq j} a_{ij}=0 \mbox{ for every } j \mbox{ and }
		\sum_{i\neq j,k} a_{ijk}=0 \mbox{ for every } j\neq k.
		\]
		As before, we distinguish three cases.
		\begin{case}[$p=p_0$]
			As we did for Aubin's result, since $a_{ij}\neq 0$ for every
			$i\neq j$, by \eqref{wijijplus} and \eqref{coeffweyl} we have
			\[
			\abs{\operatorname{W^+}_{\tilde{g}}}^2_{\tilde{g}}
			\geq 2\sum_{i<j}(\tilde{W}_{ijij}^+)^2=
			(\lambda f')^4\sum_{i<j}(a_{ij})^2>0.
			\]
		\end{case}
		\begin{case}[$p\in B_{r/2}\setminus \{p_0\}$]
			We can apply again Taylor's Theorem to conclude that
			\[
			\abs{\weyl_g^+}\leq C\cdot r + o(r^2), \mbox{ as }
			r\to 0.
			\]
			Let us suppose $\tilde{W}_{ijij}^+=\tilde{W}_{ijik}^+=0$
			for every $i\neq j\neq k$.
			By \eqref{wijikplus}, letting $r\to 0$ we have
			\[
			a_{ijk}x_jx_k\pm a_{jil}x_ix_l=0.
			\]
			More explicitly, we obtain the system
			\[
			\begin{cases}
				a_{123}x_2x_3+a_{214}x_1x_4=0\\
				a_{124}x_2x_4-a_{213}x_1x_3=0\\
				a_{134}x_3x_4+a_{312}x_1x_2=0
			\end{cases};
			\]
			by \eqref{coeffweylplus}, the system becomes
			\[
			\begin{cases}
				4x_1x_4=-15x_2x_3\\
				9x_1x_3=10x_2x_4\\
				5x_1x_2=-6x_3x_4
			\end{cases}.
			\]
			If $x_i\neq 0$ for every $i=1,2,3,4$, a straightforward computation
			shows that the system does not admit any real solution: therefore,
			the components $\tilde{W}_{ijik}^+$ cannot simultaneously vanish.
			Thus, without loss of generality, we may suppose $x_4=0$. This implies
			immediately that two out of the three remaining variables must be zero.
			Let us suppose that $x_2=x_3=x_4=0$ and $x_1\neq 0$ (the other cases
			are analogous). By $\tilde{W}_{ijij}^+=0$, for a sufficiently
			small $r$,
			\eqref{wijijplus} implies that
			\[
			a_{ij}(f')^2+b_{ij}f'f''=0.
			\]
			However, since using \eqref{coeffweyl} and \eqref{coeffweylplus} one has
			\[
			a_{13}(f')^2+b_{13}f'f''=0
			\quad\Longrightarrow\quad
			x_1^2=\dfrac{r^2}{4}\cdot\dfrac{f'}{f''},
			\]
			we get a contradiction, since, by definition of $f$, the ratio
			$f'/f''$ is negative on $B_{\frac{r}{2}}$.
		\end{case}
		\begin{case}[$p\in S\setminus B_{\frac{r}{2}}$]
			As before, let us suppose that $\tilde{W}_{ijij}^+=
			\tilde{W}_{ijik}^+=0$ for every $i\neq j\neq k$. As
			$r\to 0$, by \eqref{wijijplus} and \eqref{wijikplus} we obtain
			the system
			\[
			\begin{cases}
				W_{ijij}^+ +\frac{\lambda^2}{2}(a_{ij}(f')^2+b_{ij}f'f'')=0\\
				W_{ijik}^+ +\frac{\lambda^2}{2}(a_{ijk}x_jx_k\pm
				a_{jil}x_ix_l)f'f''=0
			\end{cases}.
			\]
			As in the proof of Theorem \ref{t-aub}, if we
			suppose that $\weyl^+$ does not identically
			vanish at $p$, possibly choosing $\lambda$ outside of a
			finite set of values, we obtain a contradiction: therefore,
			$\weyl^+=0$ at $p$, which is impossible for the
			conclusions of Case 2.
		\end{case}
		Thus,
		\[
		\abs{\weyl_{\tilde{g}}^+}_{\tilde{g}}^2>0
		\]
		on S: since $M$ is compact, we can repeat the argument presented in
		Step 2 of the proof of Theorem \ref{t-aub}
		to prove the claim.
		
		Note that the proof is analogous if we consider
		$\weyl_{\tilde{g}}^-$.
\end{proof}

Now, we prove the general condition defined in Theorem \ref{t-mixed}
\begin{proof}[Proof of Theorem \ref{t-mixed}]
	First, note that, if $t=1$, there is nothing to show: indeed
	$\weyl=\operatorname{W^+}+\operatorname{W^-}$, therefore
	Aubin's Theorem guarantees that the claim is true. If $t=0$,
	we obtain Theorem \ref{t-sd}.
	
	Now, let us suppose $t=-1$. A straightforward computation shows that
	\begin{align*}
		W_{ijij}^+-W_{ijij}^-&=W_{iji'j'}\\
		W_{ijik}^+-W_{ijik}^-&=\pm W_{iji'k'}\\
		W_{ijkl}^+-W_{ijkl}^-&=\pm W_{ijij};
	\end{align*}
	hence, we can apply again Theorem \ref{t-aub} to show the claim.
	
	Therefore, let $t\neq -1,0,1$. We consider again the deformed
	metric $\tilde{g}_t$ defined by \eqref{aubindef}, with $\phi$
	as in \eqref{phiweyl}. It is easy to obtain the system
	\small
	\begin{subequations}
		\begin{empheq}[left=\empheqlbrace]{align}
			\tilde{W}_{ijij}^+ +t\tilde{W}_{ijij}^-&=
			W_{ijij}^+ +tW_{ijij}^-+\frac{\lambda^2}{2}(1+t)
			[a_{ij}(f')^2+b_{ij}f'f'']+O(r^2) \label{mixwijij}\\
			\tilde{W}_{ijik}^+ +t\tilde{W}_{ijik}^-&=
			W_{ijik}^+ +tW_{ijik}^-+\frac{\lambda^2}{2}[(1+t)
			a_{ijk}x_jx_k\pm (1-t)a_{jil}x_ix_l]f'f''+O(r^2)
			\label{mixwijik}\\
			\tilde{W}_{ijkl}^+ +t\tilde{W}_{ijkl}^-&=
			W_{ijkl}^+ +tW_{ijkl}^-\pm\frac{\lambda^2}{2}(1-t)
			[a_{ij}(f')^2+b_{ij}f'f'']+O(r^2) \label{mixwijkl}
		\end{empheq}
	\end{subequations}
	\normalsize
	where $i\neq j\neq k\neq l$. As we did for the proof of
	Aubin's Theorem, let $p_0\in M$ be a point such that
	$\weyl_g^++t\weyl_g^-\restrict{p_0}=0$ and let
	$B_r$ be an open ball of radius $r$ and centered in $p_0$; moreover, let
	us define normal coordinates $x_1,...x_4$ such that
	$p_0=(0,0,0,0)$ and let $p\in B_r$.
	We define $\phi$ and $S=\interior{\supp}\phi$ as usual;
	finally, we choose the
	coefficients $(\alpha_1,...,\alpha_4)$ such that $a_{ij},a_{ijk}\neq 0$
	for every $i,j,k$: note that the coefficients can be chosen in such a way
	that the numbers $a_{ijk}$ have the same sign. By \eqref{coeffweylplus},
	it is easy to see that $\ul{\alpha}=(1,5/4,3/2,2)$ is a suitable choice.
	\setcounter{case}{0}
	\begin{case}[$p=p_0$]
		As usual, since $a_{ij}\neq 0$, we have that
		\[
		\tilde{W}_{ijij}^++t\tilde{W}_{ijij}^-=
		\dfrac{\lambda^2}{2}(1+t)a_{ij}(f')^2\neq 0, \qquad
		\tilde{W}_{ijkl}^++t\tilde{W}_{ijkl}^-=
		\dfrac{\lambda^2}{2}(1-t)a_{ij}(f')^2\neq 0
		\]
		at $p_0$; therefore $\weyl_{\tilde{g}_t}^+ +
		t\weyl_{\tilde{g}_t}^-\not\equiv0$ at $p_0$.
	\end{case}
	\begin{case}[$p\in B_{r/2}\setminus\{p_0\}$]
		For a sufficiently small radius $r$, we again have that
		\[
		\abs{\weyl_g^++t\weyl_g^-}\leq C\dot r
		+o(r^2), \mbox{ as } r\to 0.
		\]
		Let us suppose that $\tilde{W}_{ijkl}^++t\tilde{W}_{ijkl}^-=0$
		at $p$: therefore,
		the subsystem consisting of the equations
		of the form \eqref{mixwijik} becomes
		\[
		\begin{cases}
			(1+t)a_{123}x_2x_3+(1-t)a_{214}x_1x_4&=0\\
			(1+t)a_{124}x_2x_4-(1-t)a_{213}x_1x_3&=0\\
			(1+t)a_{134}x_3x_4+(1-t)a_{312}x_1x_2&=0
		\end{cases}.
		\]
		Let us suppose that $x_1,...,x_4\neq 0$: hence, we have
		\[
		\dfrac{1-t}{1+t}=\dfrac{a_{124}}{a_{213}}\cdot\dfrac{x_2x_4}{x_1x_3}=
		-\dfrac{a_{123}}{a_{214}}\cdot\dfrac{x_2x_3}{x_1x_4}\Rightarrow
		\dfrac{a_{124}}{a_{213}}\cdot\dfrac{x_4^2}{x_3^2}=-
		\dfrac{a_{123}}{a_{214}},
		\]
		which is impossible, since, by hypothesis, the coefficients $a_{ijk}$
		all have the same sign. Thus, at least one coordinate $x_i$ must
		vanish and, by the system above, this implies that there is just
		one coordinate of $p$ different from zero. Without loss of
		generality, we may suppose that $x_1\neq 0$. However,
		by choosing the coefficients $\alpha_1,...\alpha_4$
		in such a way that $a_{ij}$ and the coefficient of
		$x_1^2$ in $b_{ij}$ have opposite signs for
		some $i\neq j$, we get a contradiction, since
		$(f')^2$ and $f'f''$ have opposite signs on $S$: for instance, if
		$\ul{\alpha}=(1,5/4,3/2,2)$, by \eqref{coeffweyl} we have
		\[
		a_{12}=\dfrac{5}{16} \quad \mbox{ and } \quad
		b_{12}=-\dfrac{1}{3r^2}x_1^2.
		\]
		Thus, the only solution of the system is $x_1=...=x_4=0$, which
		is impossible, since $p\neq p_0$: hence, we conclude that
		$\weyl_{\tilde{g}_t}^++t\weyl_{\tilde{g}_t}^-$
		does not identically vanish at $p$.
	\end{case}
	\begin{case}[$p\in S\setminus B_{r/2}$]
		If we suppose that
		$\weyl_{\tilde{g}_t}^++t\weyl_{\tilde{g}_t}^-$
		identically vanish at $p$, as $r\to 0$ the system consisting of
		the equations \eqref{mixwijij}, \eqref{mixwijik} and \eqref{mixwijkl}
		becomes
		\[
		\begin{cases}
			0&=
			W_{ijij}^+ +tW_{ijij}^-+\frac{\lambda^2}{2}(1+t)
			[a_{ij}(f')^2+b_{ij}f'f'']\\
			0&=
			W_{ijik}^+ +tW_{ijik}^-+\frac{\lambda^2}{2}[(1+t)
			a_{ijk}x_jx_k\pm (1-t)a_{jil}x_ix_l]f'f''\\
			0&=
			W_{ijkl}^+ +tW_{ijkl}^-\pm\frac{\lambda^2}{2}(1-t)
			[a_{ij}(f')^2+b_{ij}f'f'']
		\end{cases}.
		\]
		However, if we suppose that
		$\weyl_g^++t\weyl_g^-$ does not identically
		vanish at $p$, as we did in the proofs of Theorem \ref{t-aub} and
		Theorem \eqref{t-sd}, by possibly choosing $\lambda$
		out of a finite set of values, we get a contradiction. Therefore,
		$\weyl_g^++t\weyl_g^-$ must vanish at $p$,
		which is impossible.
	\end{case}
	By the hypothesis of compactness on $M$, the claim is proven.
\end{proof}

\

\section{Proof of Theorem \ref{t-cot}}
In this section we prove Theorem \ref{t-cot}. If we use again Aubin's deformation of $g$ as described in \eqref{aubindef},
we can write the components of the Cotton tensor with respect to the deformed
metric $\tilde{g}$ as
\begin{align} \label{cottonaubin}
	\tilde{C}_{ijk}&=C_{ijk}-\dfrac{1}{w}[(\phi_k^t\phi^s+\phi_k^s\phi^t)
	R_{itjs}-(\phi_j^t\phi^s+\phi_j^s\phi^t)R_{itks}]+\\
	&-\dfrac{\phi^p}{w}\phi_{ik}\left\{R_{jp}-\dfrac{1}{w}\sq{
		\phi^t\phi^s(R_{ptjs}+\phi_{jp}\phi_{ts}-\phi_{pt}\phi_{js})-
		(\Delta\phi)\phi_{jp}+\phi_{pt}\phi_j^t}\right\}\notag+\\
	&+\dfrac{\phi^p}{w}\phi_{ij}\left\{R_{kp}-\dfrac{1}{w}\sq{
		\phi^t\phi^s(R_{ptks}+\phi_{kp}\phi_{ts}-\phi_{pt}\phi_{ks})-
		(\Delta\phi)\phi_{kp}+\phi_{pt}\phi_k^t}\right\}\notag+\\
	&+\dfrac{1}{w}[(\Delta\phi)_k\phi_{ij}-(\Delta\phi)_j\phi_{ik}+
	(\Delta\phi)\phi^sR_{sijk}-\phi_i^t\phi^sR_{stjk}
	+\phi_k^t\phi_{itj}-\phi_j^t\phi_{itk}+
	\phi^t\phi^s(R_{itjs,k}-R_{itks,j})]+\notag\\
	&+\dfrac{2\phi^p}{w^2}[\phi^t\phi^s(\phi_{kp}R_{itjs}-\phi_{jp}R_{itks})
	+\phi_{jp}((\Delta\phi)\phi_{ik}-\phi_{it}\phi_k^t)
	-\phi_{kp}((\Delta\phi)\phi_{ij}-\phi_{it}\phi_j^t)]+\notag\\
	&-\dfrac{1}{w^2}\left\{\phi^p[\phi_k^s(\phi_{ij}\phi_{sp}-
	\phi_{is}\phi_{jp})-\phi_j^t(\phi_{ik}\phi_{pt}-\phi_{it}\phi_{kp})]
	\right\}+\notag\\
	&-\dfrac{1}{w^2}\left\{\phi^p[\phi_k^s(\phi_{ij}\phi_{ps}-
	\phi_{ip}\phi_{js})-\phi_j^t(\phi_{ik}\phi_{pt}-\phi_{ip}\phi_{kt})
	]\right\}\notag+\\
	&-\dfrac{1}{w^2}\left\{\phi^s\phi^t(\phi^r(R_{rijk}\phi_{ts}-
	R_{rsjk}\phi_{it})+\phi_{tsk}\phi_{ij}-\phi_{tsj}\phi_{ik}-
	\phi_{itk}\phi_{js}+\phi_{itj}\phi_{ks})\right\}+\notag\\
	&-\dfrac{4\phi^p}{w^3}\phi^t\phi^s[\phi_{kp}(\phi_{ij}\phi_{ts}-
	\phi_{it}\phi_{js})-\phi_{jp}(\phi_{ik}\phi_{ts}-\phi_{it}\phi_{ks})]+
	\notag\\
	&-\dfrac{1}{2w(n-1)}[\phi^p\phi^qR_{pq,k}+2R_{pq}\phi^p\phi_k^q+
	2(\Delta\phi)(\Delta\phi)_k-2\phi^{pq}\phi_{pqk}](g_{ij}+\phi_i\phi_j)+
	\notag\\
	&+\dfrac{1}{2w(n-1)}[\phi^p\phi^qR_{pq,j}+2R_{pq}\phi^p\phi_j^q+
	2(\Delta\phi)(\Delta\phi)_j-2\phi^{pq}\phi_{pqj}](g_{ik}+\phi_i\phi_k)+
	\notag\\
	&-\dfrac{1}{2w^2(n-1)}\left\{2\phi^p\phi_{pk}\sq{2R_{st}\phi^s\phi^t
		-(\Delta\phi)^2+\phi_{st}\phi^{st}+\dfrac{4}{w}((\Delta\phi)\phi^s
		\phi^t\phi_{st}-\phi^r\phi_{rs}\phi^{st}\phi_t)}+\right.\notag\\
	&+\left.(\Delta\phi)_k\phi^p\phi^q\phi_{pq}+(\Delta\phi)
	\phi^p\phi^q\phi_{pqk}+2(\Delta\phi)\phi^p\phi_k^q\phi_{pq}-
	2\phi^p\phi^q\phi_p^s\phi_{sqk}-2\phi^p\phi_{pq}\phi^{qs}\phi_{sk}
	\right\}(g_{ij}+\phi_i\phi_j)+\notag\\
	&+\dfrac{1}{2w^2(n-1)}\left\{2\phi^p\phi_{pj}\sq{2R_{st}\phi^s\phi^t
		-(\Delta\phi)^2+\phi_{st}\phi^{st}+\dfrac{4}{w}((\Delta\phi)\phi^s
		\phi^t\phi_{st}-\phi^r\phi_{rs}\phi^{st}\phi_t)}+\right.\notag\\
	&+\left.(\Delta\phi)_j\phi^p\phi^q\phi_{pq}+(\Delta\phi)
	\phi^p\phi^q\phi_{pqj}+2(\Delta\phi)\phi^p\phi_j^q\phi_{pq}-
	2\phi^p\phi^q\phi_p^s\phi_{sqj}-2\phi^p\phi_{pq}\phi^{qs}\phi_{sj}
	\right\}(g_{ik}+\phi_i\phi_k)+\notag\\
	&-\dfrac{2}{n-1}(S_k\phi_i\phi_j-S_j\phi_i\phi_k).\notag
\end{align}
\begin{proof}
	Let $g$ any Riemannian metric on $M$  and consider the deformed metric $\tilde{g}$ defined in \eqref{aubindef},
	where $\phi$ is chosen as in \eqref{phiweyl}, with
	$\alpha_1,...,\alpha_n\in[1,2]$ and such that the derivatives of $f$
	satisfies the following inequalities
	\[
	f'>0, \qquad f''<0, \qquad f'''>0 \qquad \mbox{ on } [0,1)
	\]
	(for instance, we can choose \eqref{rightfunction} with a sufficiently
	large $b$).
	Let us choose a point $p_0\in M$ where
	the Cotton tensor $\cott$ of $(M,g)$ vanishes and let us
	consider again an open ball $B_r$ with normal coordinates
	centered at $p_0$; we also define $\phi$ and $S=\interior{\supp}\phi$
	as usual.
	Note that, in addition to \eqref{derphiweyl}
	and \eqref{2derphiweyl}, for a sufficiently small $r$ we have
	\begin{equation} \label{3phiweyl}
		\phi_{ijk}=\dfrac{2\lambda}{r^2}\alpha_i\sq{(
			\alpha_jx_i\delta_{jk}+\alpha_jx_j\delta_{ik}+\alpha_kx_k\delta_{ij})f''
			+\dfrac{2\alpha_j\alpha_k}{r^2}x_ix_jx_kf'''}=O\pa{\dfrac{1}{r}}.
	\end{equation}
	By \eqref{2derphiweyl} and \eqref{3phiweyl}, we obtain
	\begin{align} \label{laplphi}
		\Delta\phi&=\lambda\pa{f'\sum_{p=1}^n\alpha_p + \dfrac{2}{r^2}f''\sum_{p=1}^n
			\alpha_p^2x_p^2}\\ \label{covlaplphi}
		(\Delta\phi)_k&=\dfrac{2\lambda}{r^2}\sq{\pa{2\alpha_k^2x_k+\alpha_kx_k
				\sum_{p=1}^n\alpha_p}f''+\dfrac{2\alpha_k}{r^2}f'''\pa{\sum_{p=1}^n
				\alpha_p^2x_p^2}x_k}
	\end{align}
	As we did for $\tilde{\weyl}$ in \eqref{apprweyl},
	for sufficiently small radii we can consider the principal part of
	the transformed Cotton tensor:
	\begin{align} \label{apprcotton}
		\tilde{C}_{ijk}&=C_{ijk}+(\Delta\phi)_k\phi_{ij}-
		(\Delta\phi)_j\phi_{ik}+\phi_{tk}\phi_{itj}-\phi_{tj}\phi_{itk}+\\
		&-\dfrac{1}{n-1}[((\Delta\phi)(\Delta\phi)_k-\phi_{pq}\phi_{pqk})
		g_{ij}-((\Delta\phi)(\Delta\phi)_j-\phi_{pq}\phi_{pqj})
		g_{ik}] + O(r)\notag,
	\end{align}
	where the expression $O(r)$ contains all the terms in
	\eqref{apprcotton} whose order is the same as $r$ or higher.
	By inserting \eqref{3phiweyl}, \eqref{laplphi} and \eqref{covlaplphi}
	into \eqref{apprcotton}, we obtain
	\begin{align} \label{cottremainder}
		\tilde{C}_{iji}&=C_{iji}+
		\lambda^2\set{a_{ij}f'f''+b_{ij}\sq{f'f'''+(f'')^2}}x_j+O(r^2)\\
		\tilde{C}_{ijk}&=C_{ijk}+\lambda^2a_{ijk}x_ix_jx_k[(f'')^2+f'f''']
		+O(r),
		\nonumber
	\end{align}
	where $i\neq j\neq k$ and
	\begin{align} \label{coeffcott}
		a_{ij}&=\dfrac{2\alpha_j}{r^2}\sq{-4\alpha_i\alpha_j-
			\alpha_i\sum_{k\neq i,j}\alpha_k+\dfrac{2}{n-1}
			\pa{\alpha_j\sum_{k\neq j}\alpha_k+\sum_{k<l}\alpha_k\alpha_l}};\\
		b_{ij}&=\dfrac{4\alpha_j}{r^4}\sq{-\alpha_i\pa{\alpha_i\alpha_jx_i^2+
				\sum_{k\neq i}\alpha_k^2x_k^2}+\dfrac{1}{n-1}\sum_k\alpha_k\pa{
				\sum_{l\neq k}\alpha_l^2x_l^2}}; \nonumber\\
		a_{ijk}&=\dfrac{4\alpha_i\alpha_j\alpha_k}{r^4}(\alpha_k-\alpha_j).
		\nonumber
	\end{align}
	Note that it is sufficient to choose $\alpha_1,...,\alpha_n$ such that
	$\alpha_i\neq\alpha_j$ for every $i\neq j$ to obtain $a_{ijk}\neq 0$
	for every $i\neq j\neq k$.
	
	It is immediate to observe that, by \eqref{cottremainder}, the
	deformed cotton tensor $\cott_{\tilde{g}}$ vanish at $p_0$.
	Thus, we want to show that $\cott_{\tilde{g}}$ does not
	identically vanish on $S\setminus\{p_0\}$: by the compactness of $M$,
	we can repeat the finiteness argument used to prove Theorem
	\ref{t-aub} in order to conclude that
	the Cotton tensor $\cott_{\tilde{g}}$ does not identically
	vanish on $M\setminus\{p_0=p_0^1,...,p_0^k\}=:M\setminus\{p_1,...,p_k\}$.
	
	Now, let $p\in S$ and let us consider $\cott_{\tilde{g}}$
	at $p$.
	\setcounter{case}{0}
	\begin{case}[$p\in B_{r/2}\setminus\{p_0\}$]
		As usual, we have that
		\[
		\abs{\cott_g}\leq D\cdot r + o(r^2), \mbox{ as } r\to 0;
		\]
		if we suppose that $\tilde{C}_{iji}=\tilde{C}_{ijk}=0$ for every
		$i\neq j\neq k$, we have that
		\[
		\begin{cases}
			a_{ij}f'f''+b_{ij}[f'f'''+(f'')^2]x_j&=0\\
			a_{ijk}x_ix_jx_k[(f'')^2+f'f''']&=0
		\end{cases}
		\]
		for a sufficiently small $r$.
		By the properties of $f$ and our choice of $\alpha_1,...,
		\alpha_n$, we have that $x_ix_jx_k=0$ for every $i\neq j\neq k$, which
		implies that at most two coordinates of $p$ are not zero.
		
		Therefore, let us suppose that $x_i,x_j\neq 0$. By hypothesis,
		$\tilde{C}_{iji}=\tilde{C}_{jij}=0$: hence,
		by \eqref{cottremainder} and \eqref{coeffcott} we obtain the following
		equations
		\begin{align*}
			0&=\sq{-4\alpha_i\alpha_j-
				\alpha_i\sum_{k\neq i,j}\alpha_k+\dfrac{2}{n-1}
				\pa{\alpha_j\sum_{k\neq j}\alpha_k+\sum_{k<l}\alpha_k\alpha_l}}
			f'f''+\\
			&+\dfrac{2}{r^2}\sq{-\alpha_i\pa{\alpha_i\alpha_jx_i^2+
					\alpha_j^2x_j^2}+\dfrac{1}{n-1}\sum_k\alpha_k\pa{
					\sum_{l\neq k}\alpha_l^2x_l^2}}\sq{(f'')^2+f'f'''};\\
			0&=\sq{-4\alpha_i\alpha_j-
				\alpha_j\sum_{k\neq i,j}\alpha_k+\dfrac{2}{n-1}
				\pa{\alpha_i\sum_{k\neq i}\alpha_k+\sum_{k<l}\alpha_k\alpha_l}}
			f'f''+\\
			&+\dfrac{2}{r^2}\sq{-\alpha_j\pa{\alpha_i\alpha_jx_j^2+
					\alpha_i^2x_i^2}+\dfrac{1}{n-1}\sum_k\alpha_k\pa{
					\sum_{l\neq k}\alpha_l^2x_l^2}}\sq{(f'')^2+f'f'''};
		\end{align*}
		subtracting the second equation from the first, it is easy to obtain
		\[
		(\alpha_j-\alpha_i)\sum_{k\neq i,j}\alpha_k+\dfrac{2}{n-1}
		\pa{\alpha_j\sum_{k\neq j}\alpha_k-\alpha_i\sum_{k\neq i}\alpha_k}=0
		\Leftrightarrow
		\dfrac{n-3}{n-1}(\alpha_j-\alpha_i)\sum_{k\neq i,j}\alpha_k=0,
		\]
		which is impossible, since $\alpha_i\neq\alpha_j$ by hypothesis.
		This implies that exactly one coordinate of $p$ is different from
		zero (say, $x_j$). Since $n\geq 4$, if $i\neq t\neq j$, by
		$\tilde{C}_{iji}=\tilde{C}_{tjt}=0$ we obtain
		\begin{align*}
			0&=\sq{-4\alpha_i\alpha_j-\alpha_i\sum_{k\neq i,j}\alpha_k
				+\dfrac{2}{n-1}\pa{\alpha_j\sum_{k\neq j}
					\alpha_k+\sum_{k<l}\alpha_k\alpha_l}}f'f''+\\
			&+\dfrac{2}{r^2}\alpha_j^2x_j^2\sq{-\alpha_i+
				\dfrac{1}{n-1}\sum_{k\neq j}\alpha_k}\sq{(f'')^2+f'f'''};\\
			0&=\sq{-4\alpha_t\alpha_j-\alpha_t\sum_{k\neq t,j}\alpha_k
				+\dfrac{2}{n-1}\pa{\alpha_j\sum_{k\neq j}
					\alpha_k+\sum_{k<l}\alpha_k\alpha_l}}f'f''+\\
			&+\dfrac{2}{r^2}\alpha_j^2x_j^2\sq{-\alpha_t+
				\dfrac{1}{n-1}\sum_{k\neq j}\alpha_k}\sq{(f'')^2+f'f'''}.
		\end{align*}
		It is not hard to see that, for a suitable choice of
		$\alpha_1\neq...\neq\alpha_n$,
		the coefficients of
		$[(f'')^2+f'f''']$ in the equations do not vanish: this allows us
		to compute $x_j^2$ as
		\[
		x_j^2=\dfrac{r^2\sq{4\alpha_i\alpha_j+\alpha_i\sum_{k\neq i,j}\alpha_k
				-\dfrac{2}{n-1}\pa{\alpha_j\sum_{k\neq j}
					\alpha_k+\sum_{k<l}\alpha_k\alpha_l}}f'f''}
		{2\alpha_j^2\sq{-\alpha_i+
				\dfrac{1}{n-1}\sum_{k\neq j}\alpha_k}\sq{(f'')^2+f'f'''}}.
		\]
		However, inserting this into the other equation, we obtain
		\begin{align*}
			&\sq{4\alpha_t\alpha_j+\alpha_t\sum_{k\neq t,j}\alpha_k
				-\dfrac{2}{n-1}\pa{\alpha_j\sum_{k\neq j}
					\alpha_k+\sum_{k<l}\alpha_k\alpha_l}}\sq{-\alpha_i+
				\dfrac{1}{n-1}\sum_{k\neq j}\alpha_k}=\\
			=&\sq{4\alpha_i\alpha_j+\alpha_i\sum_{k\neq i,j}\alpha_k
				-\dfrac{2}{n-1}\pa{\alpha_j\sum_{k\neq j}
					\alpha_k+\sum_{k<l}\alpha_k\alpha_l}}\sq{-\alpha_t+
				\dfrac{1}{n-1}\sum_{k\neq j}\alpha_k},
		\end{align*}
		which implies
		\begin{align*}
			&\dfrac{4}{n-1}(\alpha_t-\alpha_i)\sum_{k\neq j}\alpha_k+
			\alpha_i\alpha_t(\alpha_t-\alpha_i)+\\
			+&\dfrac{1}{n-1}(\alpha_t-\alpha_i)
			\pa{\sum_{k\neq i,j,t}\alpha_k}\pa{\sum_{l\neq j}\alpha_l}+
			\dfrac{2}{(n-1)^2}(\alpha_t-\alpha_i)\pa{\alpha_j\sum_{k\neq j}
				\alpha_k+\sum_{k<l}\alpha_k\alpha_l}=0
		\end{align*}
		and this is clearly impossible. Since $p\neq p_0$, we have that
		the Cotton tensor $\cott_{\tilde{g}}$ cannot identically
		vanish at $p$.
	\end{case}
	\begin{case}[$p\in S\setminus B_{r/2}$]
		As usual, let us suppose that $\cott_{\tilde{g}}$ identically
		vanishes at $p$. If $\cott$ does not vanish at $p$, we
		can exploit the argument of Theorem \ref{t-aub} to conclude that,
		if we possibly choose $\lambda$ out of a finite set of values,
		this is impossible. Therefore, $\cott\equiv 0$ at $p$,
		which is a contradiction, by the proof of Case 1; hence,
		$\cott_{\tilde{g}}$ does not vanish at $p$.
	\end{case}
	The hypothesis of compactness on $M$ proves the claim.
\end{proof}


\section{Proof of Theorem \ref{t-bac}}

In this section, we focus on four-dimensional manifolds and we prove Theorem \ref{t-bac}. If $n=4$, the Bach tensor acquires two additional properties: it is
conformally invariant and divergence-free (see \cite{cmbook},
Section 1.4 and Section 2.2.2).

\begin{proof}
	As we did in the proof of Theorem \ref{t-aub}, let $g$ any Riemannian metric on $M$  and let $p_0\in M$ such that
	$\bach_g$ vanishes and let $B_r$ an open ball
	of radius $r$ and centered in $p_0$.
	Let us choose normal coordinates $x_1,...,x_4$ such that $p_0=(0,0,0,0)$
	and let us define the function $\phi$ as in \eqref{phiweyl} and
	$S=\interior{\supp}\phi$ as usual, with
	$f$ defined as in \eqref{rightfunction}.
	We know that $f\in C^{\infty}({[0,+\infty)})$: therefore,
	$\phi\in C^{\infty}(M)$ and it vanishes outside $S$. Moreover,
	for a sufficiently large $b$, the function
	$f$ satisfies the following inequalities
	\[
	f'>0, \quad f''<0, \quad f'''>0, \quad f^{IV}<0
	\quad\mbox{ on } [0,1).
	\]
	By \eqref{3phiweyl} and \eqref{laplphi}, we obtain the following
	additional expressions:
	\begin{align} \label{4phiweyl}
		\phi_{ijkt}&=\dfrac{2}{r^2}\lambda\alpha_i\left\{\dfrac{4}{r^4}
		\alpha_j\alpha_k\alpha_tx_ix_jx_kx_tf^{IV}+
		(\alpha_k\delta_{kt}\delta_{ij}+\alpha_j\delta_{jt}\delta_{ik}
		+\alpha_j\delta_{it}\delta_{jk})f''\right.+\\
		&+\left.\dfrac{2}{r^2}\sq{\alpha_j\alpha_k(\delta_{it}x_jx_k+
			\delta_{jt}x_ix_k+\delta_{kt}x_ix_j)
			+\alpha_tx_t(\delta_{ij}\alpha_kx_k+\delta_{ik}\alpha_jx_j
			+\delta_{jk}\alpha_jx_i)}f'''\right\}.
		\notag
	\end{align}
	\begin{align} \label{hesslaplphi}
		(\Delta\phi)_{jk}&=\dfrac{2\lambda\alpha_j}{r^2}\left\{\pa{2\alpha_j+
			\sum_p\alpha_p}f''\delta_{jk}+\right.\\
		&+\left.\dfrac{2}{r^2}\sq{\alpha_k\pa{
				2\alpha_j+\sum_p\alpha_p}x_jx_k+2\alpha_k^2x_jx_k+\sum_p\alpha_p^2x_p^2
			\delta_{jk}}f'''+\dfrac{4\alpha_k}{r^4}\pa{\sum_p\alpha_p^2x_p^2}
		x_jx_kf^{IV}\right\}. \notag
	\end{align}
	\begin{align} \label{bilaplphi}
		(\Delta\phi)_{kk}&=\dfrac{2\lambda}{r^2}\left[\pa{2\sum_p\alpha_p^2+
			\pa{\sum_q\alpha_q}^2}f''+\dfrac{4}{r^2}\pa{2\sum_p\alpha_p^3x_p^2+
			\sum_p\alpha_p\sum_q\alpha_q^2x_q^2}f'''+\right.\\
		&+\left.\dfrac{4}{r^4}\pa{\sum_p\alpha_p^2x_p^2}^2f^{IV}\right].
		\notag
	\end{align}
	Note that
	\[
	\phi_{ijkt}=O\pa{\dfrac{1}{r^2}}
	\]
	as $r\to 0$.
	We consider the principal part of
	the transformed Bach tensor:
	by \eqref{cottonaubin}, \eqref{apprcotton} and
	the definition of the Bach tensor, we obtain
	\begin{align} \label{apprbach}
		\tilde{B}_{ij}=B_{ij}&+
		(\Delta\phi)_{kk}\phi_{ij}+(\Delta\phi)_k\phi_{ijk}-
		(\Delta\phi)_{jk}\phi_{ik}-(\Delta\phi)_j\phi_{ikk}+\\
		&+\phi_{tkk}\phi_{itj}+\phi_{tk}\phi_{itjk}-
		\phi_{tjk}\phi_{itk}-\phi_{tj}\phi_{itkk}+\notag\\
		&-\dfrac{1}{n-1}\left[((\Delta\phi)_k(\Delta\phi)_k+(\Delta\phi)
		(\Delta\phi)_{kk}-\phi_{pqk}\phi_{pqk}-\phi_{pq}\phi_{pqkk})\delta_{ij}
		+\right.\notag\\
		&\left.-((\Delta\phi)_i(\Delta\phi)_j+(\Delta\phi)(\Delta\phi)_{ji}
		-\phi_{pqi}\phi_{pqj}-\phi_{pq}\phi_{pqji})\right]+O(1)\notag,
	\end{align}
	where $O(1)$ is the usual ``remainder'' term. Note that,
	as $r\to 0$, the terms given by $\tilde{R}_{kl}\tilde{W}_{ijkl}$
	in the definition of the Bach tensor \eqref{bach}
	do not appear in \eqref{apprbach}, since their order is lower than
	the order of $\tilde{C}_{ijk,k}$; however, as we did for the Cotton
	tensor, we make explicit the coefficients of $\bach_g$, since
	they do not depend on $f$ (and, therefore, they do not \emph{a priori}
	vanish as the argument of $f$ goes to $1$).
	
	Inserting \eqref{2derphiweyl}, \eqref{laplphi}, \eqref{3phiweyl},
	\eqref{4phiweyl} and \eqref{hesslaplphi} into \eqref{apprbach},
	for a sufficiently small radius $r$ we obtain
	the following expression for the Bach tensor:
	\footnotesize
	\begin{align} \label{apprbach4}
		\tilde{B}_{ij}=B_{ij}&+\dfrac{2\lambda^2}{3r^2}\set{\alpha_i
			\sq{8\sum_p\alpha_p^2+4\sum_q\alpha_q\pa{\sum_t\alpha_t-\alpha_i}
				-8\alpha_i^2}-\sum_p\alpha_p\sq{\sum_q\alpha_q^2+
				\pa{\sum_t\alpha_t}^2}+2\sum_p\alpha_p^3}f'f''\delta_{ij}+\\
		&+\dfrac{4\lambda^2}{3r^4}\left\{4\sum_p\alpha_p^4x_p^2+
		\pa{14\alpha_i-3\sum_p\alpha_p}
		\sum_q\alpha_q^3x_q^2+\right.\notag\\
		&+\left.\sum_p\alpha_p^2x_p^2
		\sq{\alpha_i\pa{7\sum_q\alpha_q-6\alpha_i}+\sum_t\alpha_t^2-
			2\pa{\sum_r\alpha_r}^2}\right\}\sq{f'f'''+\pa{f''}^2}\delta_{ij}+\notag\\
		&+\dfrac{8\lambda^2}{3r^6}\sum_p\alpha_p^2x_p^2
		\sq{\sum_q\alpha_q^3x_q^2+\sum_q\alpha_q^2x_q^2\pa{3\alpha_i-
				\sum_t\alpha_t}}\pa{f'f^{IV}+3f''f'''}\delta_{ij}\notag+\\
		&+\dfrac{4\lambda^2\alpha_i\alpha_j}{3r^4}x_ix_j\sq{2\sum_p\alpha_p^2
			+\pa{\sum_q\alpha_q}^2-2\pa{\alpha_i^2+\alpha_j^2+6\alpha_i\alpha_j}
			-\pa{\alpha_i+\alpha_j}\sum_t\alpha_t}\sq{f'f'''+\pa{f''}^2}+\notag\\
		&+\dfrac{8\lambda^2\alpha_i\alpha_j}{3r^6}x_ix_j\sq{
			2\sum_p\alpha_p^3x_p^2-\pa{3\alpha_i+3\alpha_j-\sum_q\alpha_q}
			\sum_t\alpha_t^2x_t^2}\pa{f'f^{IV}+3f''f'''}+O(1)\notag.
	\end{align}
	\normalsize
	Let
	\[
	A:=f'f'', \qquad B:=f'f'''+(f'')^2, \qquad C:=f'f^{IV}+3f''f'''
	\]
	and let us choose $(\alpha_1,\alpha_2,\alpha_3,\alpha_4)=
	(1,\frac{5}{4},\frac{3}{2},2)$. For $i\neq j$, we obtain the following
	equations
	\begin{align} \label{bachsyst1}
		\tilde{B}_{12}&=B_{12}+\dfrac{5\lambda^2}{3r^4}
		\sq{\dfrac{2}{r^2}\pa{x_1^2+\dfrac{75}{32}x_2^2+
				\dfrac{9}{2}x_3^2+12x_4^2}C+\dfrac{141}{8}B}x_1x_2\\
		\tilde{B}_{13}&=B_{13}+\dfrac{2\lambda^2}{r^4}
		\sq{\dfrac{2}{r^2}\pa{\dfrac{1}{4}x_1^2+\dfrac{75}{64}x_2^2+
				\dfrac{45}{16}x_3^2+9x_4^2}C+\dfrac{189}{16}B}x_1x_3\notag\\
		\tilde{B}_{14}&=B_{14}+\dfrac{8\lambda^2}{3r^4}
		\sq{-\dfrac{2}{r^2}\pa{\dfrac{5}{4}x_1^2+\dfrac{75}{64}x_2^2+
				\dfrac{9}{16}x_3^2-3x_4^2}C-\dfrac{9}{16}B}x_1x_4\notag\\
		\tilde{B}_{23}&=B_{23}+\dfrac{5\lambda^2}{2r^4}
		\sq{\dfrac{2}{r^2}\pa{-\dfrac{1}{2}x_1^2+
				\dfrac{9}{8}x_3^2+6x_4^2}C+\dfrac{19}{4}B}x_2x_3\notag\\
		\tilde{B}_{24}&=B_{24}+\dfrac{10\lambda^2}{3r^4}
		\sq{-\dfrac{2}{r^2}\pa{x_1^2+\dfrac{75}{32}x_2^2+
				\dfrac{9}{4}x_3^2}C-\dfrac{73}{8}B}x_2x_4\notag\\
		\tilde{B}_{34}&=B_{34}+\dfrac{4\lambda^2}{r^4}
		\sq{-\dfrac{2}{r^2}\pa{\dfrac{11}{4}x_1^2+\dfrac{225}{64}x_2^2+
				\dfrac{63}{16}x_3^2+3x_4^2}C-\dfrac{287}{16}B}x_3x_4; \notag
	\end{align}
	for $i=j$, we have the additional expressions
	\begin{align} \label{bachsyst2}
		\tilde{B}_{11}=B_{11}&-\dfrac{323\lambda^2}{12r^2}A+
		\dfrac{\lambda^2}{3r^4}\pa{\dfrac{7}{2}x_1^2-\dfrac{4175}{32}x_2^2
			-\dfrac{2727}{16}x_3^2-217x_4^2}B+\\
		&+\dfrac{8\lambda^2}{3r^6}\left[
		\pa{x_1^2+\dfrac{25}{16}x_2^2+\dfrac{9}{4}x_3^2+4x_4^2}
		\pa{-\dfrac{7}{4}x_1^2-\dfrac{75}{32}x_2^2-\dfrac{45}{16}x_3^2
			-3x_4^2}+\right.\notag\\
		&+\left.x_1^2\pa{\dfrac{7}{4}x_1^2+\dfrac{225}{64}x_2^2+
			\dfrac{99}{16}x_3^2+15x_4^2}\right]C \notag\\
		\tilde{B}_{22}=B_{22}&-\dfrac{41\lambda^2}{6r^2}A+
		\dfrac{\lambda^2}{r^4}\pa{-\dfrac{97}{6}x_1^2+\dfrac{75}{24}x_2^2
			-21x_3^2+\dfrac{2}{3}x_4^2}B+\notag\\
		&+\dfrac{\lambda^2}{3r^6}\left[
		\pa{8x_1^2+\dfrac{25}{2}x_2^2+18x_3^2+32x_4^2}
		\pa{-x_1^2-\dfrac{75}{64}x_2^2-\dfrac{9}{8}x_3^2}+\right.\notag\\
		&+\left.x_2^2\pa{\dfrac{25}{8}x_1^2+\dfrac{1875}{128}x_2^2+
			\dfrac{1125}{32}x_3^2+\dfrac{225}{2}x_4^2}\right]C \notag\\
		\tilde{B}_{33}=B_{33}&+\dfrac{53\lambda^2}{6r^2}A+
		\dfrac{\lambda^2}{r^4}\pa{-\dfrac{43}{12}x_1^2+\dfrac{25}{24}x_2^2
			+\dfrac{39}{8}x_3^2+\dfrac{209}{3}x_4^2}B+\notag\\
		&+\dfrac{2\lambda^2}{r^6}\left[
		\pa{\dfrac{4}{3}x_1^2+\dfrac{25}{12}x_2^2+3x_3^2+\dfrac{16}{3}x_4^2}
		\pa{-\dfrac{1}{4}x_1^2+\dfrac{9}{16}x_3^2
			+3x_4^2}+\right.\notag\\
		&+\left.x_3^2\pa{-\dfrac{15}{4}x_1^2-\dfrac{225}{64}x_2^2-
			\dfrac{27}{16}x_3^2+9x_4^2}\right]C \notag\\
		\tilde{B}_{44}=B_{44}&+\dfrac{299\lambda^2}{12r^2}A+
		\dfrac{\lambda^2}{r^4}\pa{\dfrac{223}{12}x_1^2+\dfrac{3775}{96}x_2^2
			+\dfrac{1167}{16}x_3^2+2x_4^2}B+\notag\\
		&+\dfrac{8\lambda^2}{3r^6}\left[
		\pa{x_1^2+\dfrac{25}{16}x_2^2+\dfrac{9}{4}x_3^2+4x_4^2}
		\pa{\dfrac{5}{4}x_1^2+\dfrac{75}{32}x_2^2+\dfrac{63}{16}x_3^2
			+9x_4^2}+\right.\notag\\
		&-\left.x_4^2\pa{17x_1^2+\dfrac{375}{16}x_2^2+
			\dfrac{117}{4}x_3^2+36x_4^2}\right]C \notag
	\end{align}
	Of course, the equations in \eqref{bachsyst2} cannot be all independent,
	since the Bach tensor is trace-free.
	
	As we did for Theorem \ref{t-aub}, we consider three cases.
	\setcounter{case}{0}
	\begin{case}[$p=p_0$]
		In our local coordinates, $p_0=(0,0,0,0)$; therefore, since
		$B_g\not\equiv0$ in $p_0$ and $A<0$ on $B_r$, by \eqref{bachsyst1} and
		\eqref{bachsyst2} we obtain
		\[
		\abs{\bach_{\tilde{g}}}_{\tilde{g}}^2=
		2\sum_{i=1}^4\tilde{B}_{ii}^2=CA^2>0,
		\]
		where $C=\frac{105845\lambda^4}{36r^4}$.
	\end{case}
	\begin{case}[$p\in B_{r/2}\setminus\{p_0\}$]
		In this case, we have again that
		\[
		|B_g|\leq C\cdot r + o(r^2), \mbox{ as } r\to 0.
		\]
		Thus, we may consider just the principal parts in the system defined by
		\eqref{bachsyst1} and \eqref{bachsyst2}.
		
		Let us suppose that $\tilde{B}_{ij}=0$ for every $i,j$ at
		$p=(x_1,x_2,x_3,x_4)$. We want to show that the only solution
		of the system is given by $x_i=0$ for every $i$, which leads to
		a contradiction for the previous argument.
		
		If we suppose that $x_i\neq 0$ for every $i$, we have that, for instance,
		\[
		B=-\dfrac{16}{141r^2}\pa{x_1^2+\dfrac{75}{32}x_2^2+\dfrac{9}{2}x_3^2+
			12x_4^2}C
		\]
		by the first equation in \eqref{bachsyst1}. Since $B>0$ and $C<0$ in
		$B_r$ and $x_1,...,x_4\neq 0$,
		inserting this into the other equations in \eqref{bachsyst1},
		we obtain a system of
		five equations in the variables $x_1,...,x_4$: a straightforward computation
		shows that this system admits only the trivial solution and,
		therefore, one of the variables $x_1,...,x_4$ must be zero.
		
		Now, let us suppose that $x_i\neq 0$ for at least two indices $i$. If
		$x_i\neq 0$ for one index $i$, by \eqref{bachsyst1} and \eqref{bachsyst2}
		we obtain a system of $5$ independent
		equations in $x_j,x_k,x_l$, where $j,k,l\neq i$: by an analogous argument,
		we can show that the system admits no solutions, which implies that
		at least two variables $x_i$ and $x_j$ must be zero.
		In this case, expressing $B$ in
		terms of $C$ as before, by \eqref{bachsyst2} we can express $A$ in
		terms of $C$ as well and, therefore, obtain two independent equations
		in $x_k,x_l$; however, by our choice of the coefficients $\alpha_1,...,
		\alpha_4$, the system is once again inconsistent.
		
		Therefore, as
		in the proof of Theorem \ref{t-aub}, we obtain that
		exactly one variable $x_i$ is different from zero. Let us suppose that,
		for instance, $x_1\neq 0$. By \eqref{bachsyst2}, we have that
		\[
		\tilde{B}_{11}=-\dfrac{323\lambda^2}{12r^2}A+\dfrac{7\lambda^2}{6r^4}x_1^2B
		>0,
		\]
		since $A<0$ and $B>0$ on $S$. Thus, the system admits no solution. The
		other cases can be shown in an analogous way. Hence, we conclude that
		$\abs{\bach_{\tilde{g}}}_{\tilde{g}}^2$ must be
		strictly positive at $p$.
		
		We also point out that the same system was solved \emph{via} technical
		computing through Wolfram Mathematica (see Appendix \ref{solbach} for the code). Also note
		that the system in the Appendix is more general than the one we are
		considering in this proof: indeed, we showed that the system
		\eqref{bachsyst1}+\eqref{bachsyst2}, with $B_{ij}=0$, would admit
		no real solutions even if $A$, $B$ and $C$ were free real parameters satisfying
		$A,B,C\neq 0$.
	\end{case}
	\begin{case}[$p\in S\setminus B_{r/2}$]
		In this case, we need to consider the components of the Bach tensor
		$\bach_g$ in \eqref{bachsyst1} and \eqref{bachsyst2}.
		
		If $\bach_g\equiv0$ at $p$, we can immediately conclude that
		$\abs{\bach_{\tilde{g}}}_{\tilde{g}}^2>0$ at $p$,
		by the proof of Case 2.
		Thus, let us suppose that $\tilde{B}_{ij}=0$ at $p$ for every $i,j$ and
		that $\abs{\bach_g}_g^2>0$ at $p$. In particular, we
		may suppose that $B_{12}\neq 0$ at $p$. By the first equation
		in \eqref{bachsyst1}, we obtain that
		\[
		\lambda^2=-\dfrac{3r^4B_{12}}{5\sq{\dfrac{2}{r^2}\pa{x_1^2+
					\dfrac{75}{32}x_2^2+\dfrac{9}{2}x_3^2+12x_4^2}C+\dfrac{141}{8}B}x_1x_2}
		\]
		at $p$. However, we may choose $\lambda_1\in\mathbb{R}$ such that
		$\lambda_1^2\neq\lambda^2$ in \eqref{phiweyl}, since $\lambda$ is
		a free parameter: if we repeat the argument of the proof with
		$\lambda_1$ instead of $\lambda$, we get a contradiction and, therefore,
		we conclude that $B_{12}=0$ at $p$.
		
		Now, if $B_{13}\neq 0$ at $p$, the second equation in \eqref{bachsyst1}
		implies that
		\[
		\lambda_1^2=-\dfrac{r^4B_{13}}{2\sq{\dfrac{2}{r^2}\pa{\dfrac{1}{4}x_1^2+
					\dfrac{75}{64}x_2^2+\dfrac{45}{16}x_3^2+9x_4^2}C
				+\dfrac{189}{16}B}x_1x_3};
		\]
		again, possibly choosing $\lambda_2$ such that
		$\lambda_2^2\neq\lambda_1^2$, we obtain that $B_{13}=0$ at $p$.
		Iterating this argument for every component $B_{ij}$, we conclude that,
		possibly choosing $\ol{\lambda}$ outside a finite set $\{
		\lambda,...,\lambda_k\}$, the components $B_{ij}$ must all vanish
		at $p$. Therefore, we repeat the argument of Case 2 to conclude that
		\[
		\abs{\bach_{\tilde{g}}}_{\tilde{g}}^2>0 \mbox{ at } p.
		\]
	\end{case}
	Now, as in Step 2 of
	the proof of Theorem \ref{t-aub}, since $M$ is compact, we
	can deform the metric $g$ on a finite cover of $M$: using the argument of
	Remark \ref{weyl1}, the claim is proven.
\end{proof}
\begin{rem}
	Observe that, even if we did not obtain the full expression of
	the transformed Bach tensor, it can be easily seen that, once we
	fix a point $p\in S$, the quantity $\tilde{B}_{ij}-B_{ij}$, up to
	multiplying for a suitable power of $w$, is
	indeed a polynomial of finite degree in $\lambda$.
\end{rem}
\begin{rem}
	As we recalled in the Introduction, when $\operatorname{dim}M=4$,
	Bach-flatness is a necessary condition for $(M,g)$ to be an Einstein
	manifold; therefore, an immediate consequence of Theorem \ref{t-bac} is
	that, given a smooth manifold $M$ of dimension four, one can always
	choose a conformal class $[g]$ of Riemannian metrics which contains
	no Einstein metrics. In fact, we can say more:
	since we found a quadruple $\alpha_1,...,\alpha_4$ such that
	the system of equations \eqref{bachsyst1}+\eqref{bachsyst2}
	admits no solutions, there exists an open neighborhood $U_{\ul{\alpha}}$ of
	$\ul{\alpha}=(\alpha_1,...,\alpha_4)$
	in $Q:=[1,2]\times[1,2]\times[1,2]\times[1,2]$ such that, for every
	$\ul{\alpha}'\in U_{\ul{\alpha}}$, the system admits no solutions on $M$.
	Therefore, there exist infinitely-many conformal classes of
	Riemannian metrics on $M$ which contain no Einstein metrics.
	
	Although we did not prove it in this paper, we expect that,
	given any Riemannian metric $g$ on $M$, the
	subset
	\[
	Q':=\left\{\alpha\in Q: \abs{\bach_{g_\alpha}}_{g_\alpha}^2\equiv 1,
	\mbox{ where } g_\alpha=g+d\phi_{\alpha}\otimes d\phi_\alpha \mbox{ and
	} \phi_{\alpha} \mbox{ is defined as in \eqref{phiweyl} } \right\}
	\]
	is such that $Q\setminus Q'$ has Lebesgue measure zero in $Q$.
\end{rem}

\appendix
\section{Solutions of the systems \eqref{bachsyst1} and \eqref{bachsyst2}
	in the homogeneous case} \label{solbach}
\includegraphics[width=15cm, height=17cm]{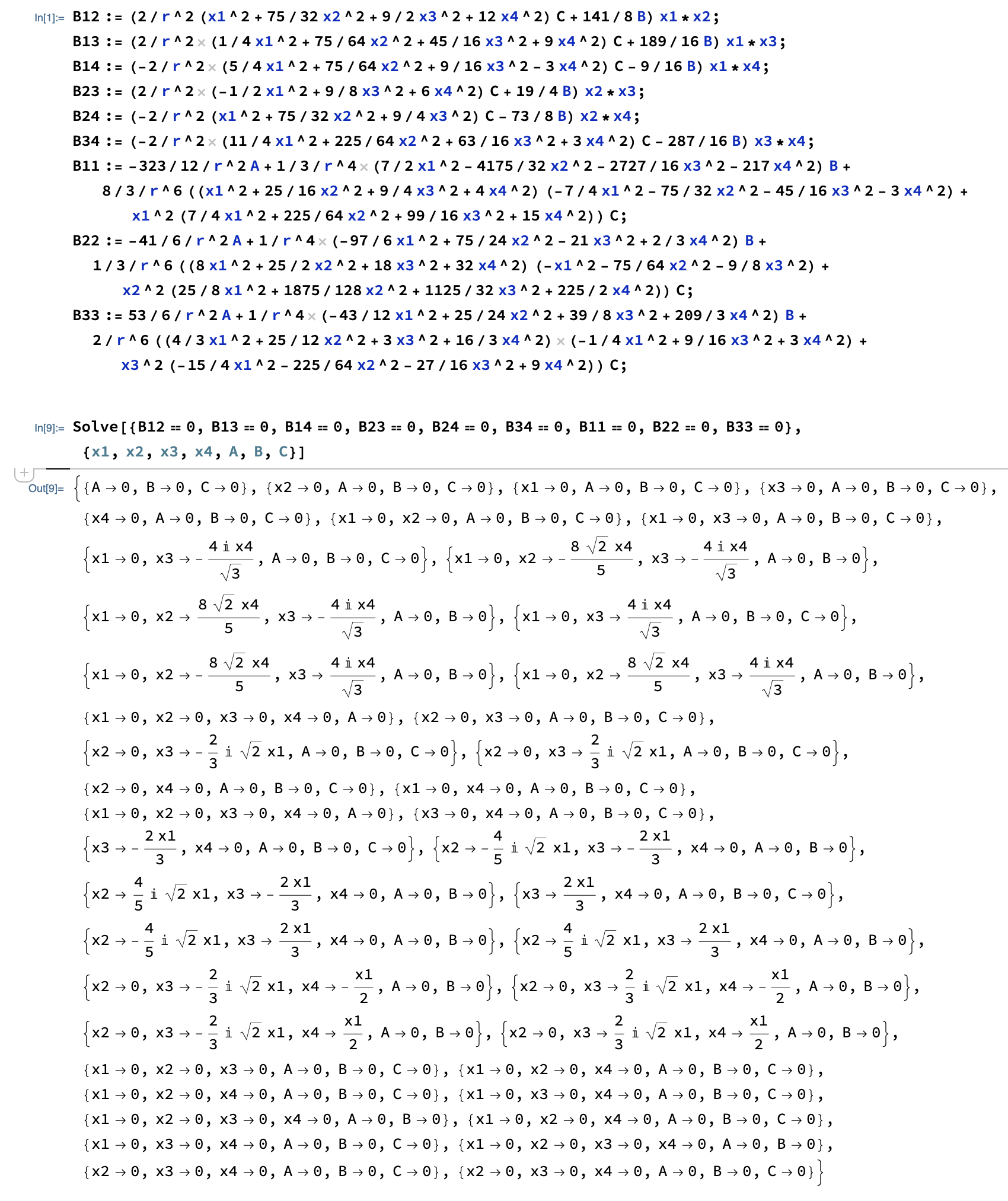}
\bibliographystyle{abbrv}
\bibliography{Aubin_biblio}
\end{document}